\documentclass[reqno,A4paper]{amsart}

\usepackage{amsmath}
\usepackage{amssymb}
\usepackage{amsthm}
\usepackage{enumerate}
\usepackage{mathrsfs} 
\usepackage{eqlist}
\usepackage{array}

\usepackage[english]{babel}                             
\usepackage[latin1]{inputenc}
\usepackage{comment}
\usepackage{ae}
\usepackage[leqno]{amsmath}  

\usepackage[dotinlabels]{titletoc}     

\usepackage{amssymb,latexsym,verbatim} 
\usepackage{amstext}
\usepackage{amsfonts}
\usepackage{amsbsy}
\usepackage{amsopn}
\usepackage{enumerate}
\usepackage{txfonts}
\usepackage{amsmath, amsthm, amssymb}

\usepackage{color}
\usepackage[usenames,dvipsnames,svgnames,table]{xcolor}

\allowdisplaybreaks
\newtheorem{theorem}{Theorem}
\newtheorem{definition}[theorem]{Definition}
\newtheorem{lemma}[theorem]{Lemma}

\newtheorem{proposition}[theorem]{Proposition}
\newtheorem{remark}[theorem]{Remark}
\numberwithin{equation}{section}
\numberwithin{theorem}{section}

\renewcommand{\epsilon}{\varepsilon}
\renewcommand{\rho}{\varrho}
\renewcommand{\widetilde}{\tilde}
\renewcommand{\Phi}{\phi}

\DeclareMathOperator*{\esssup}{ess\,sup}
\DeclareMathOperator*{\essinf}{ess\,inf}

\numberwithin{equation}{section}

\begin{document}

\title[Weak Harnack estimates for supersolutions to parabolic equations]%
{Weak Harnack estimates for supersolutions to doubly degenerate parabolic equations}
\author[Qifan Li]%
{Qifan Li*}

\newcommand{\acr}{\newline\indent}

\address{\llap{*\,}Department of Mathematics\acr
                   School of Sciences\acr
                   Wuhan University of Technology\acr
                   430070, 122 Luoshi Road,
                   Wuhan, Hubei\acr
                   P. R. China}
\email{qifan\_li@yahoo.com, qifan\_li@whut.edu.cn}


\thanks{This work was supported by the Independent Innovation Foundation of Wuhan University of Technology, Grant No. 20410685.}

\subjclass[2010]{Primary 35K65, 35K92, 35B65; Secondary 35K59, 35B45.} 
\keywords{Parabolic potential theory, weak supersolutions, weak Harnack estimates, doubly degenerate parabolic equations, expansion of positivity.}

\begin{abstract}
We establish weak Harnack inequalities for positive, weak supersolutions to certain doubly degenerate parabolic equations. The prototype of this kind of equations is
$$\partial_tu-\operatorname{div}|u|^{m-1}|Du|^{p-2}Du=0,\quad p>2,\quad m+p>3.$$
Our proof is based on Caccioppoli inequalities, De Giorgi's estimates and Moser's iterative method.
\end{abstract}

\maketitle

\section{Introduction}
Harnack inequality for non-negative solutions of the linear parabolic equations was established by Moser \cite{moser}. Furthermore, Aronson and Serrin \cite{AS}, Trudinger \cite{Tru} and Ivanov \cite{Ivanov} independently extended Moser's result to the quasilinear case. Later DiBenedetto \cite{Di} established intrinsic Harnack inequality for weak solutions to the parabolic $p$-Laplace equations. The proof is based on the idea of time-intrinsic geometry, maximum principle and comparison to Barenblatt solutions. Recently, DiBenedetto, Gianazza and Vespri \cite{DiBenedetto,DGV} extended this result to the parabolic equations with general quasi-linear structure.

The issue of the
Harnack inequality for weak supersolutions to quasi-linear parabolic equations was settled by Trudinger \cite{Tru}. However, the weak supersolutions may not be continuous and they do not, in general, satisfy an intrinsic Harnack inequality (see for instance \cite{lindqvist}). In \cite{Kuusi}, Kuusi proved that any non-negative weak supersolutions to a certian evolutionary $p$-Laplace equations satisfy an integral form of the intrinsic Harnack inequality. Similar result for the degenerate porous medium equations was announced by DiBenedetto, Gianazza and Vespri \cite{DiBenedetto} and later proved by Lehtel\"a \cite{Le}. This kind of weak Harnack inequality was used by Kuusi, Lindqvist and Parviainen \cite{KPL} to the study of summability of unbounded supersolutions.

In this work, we are interested in the doubly degenerate parabolic equations whose prototype
is
\begin{equation}\label{prototype}\frac{\partial u}{\partial t}-\operatorname{div}|u|^{m-1}|Du|^{p-2}Du=0.\end{equation}
This equation models the filtration of a polytropic non-Newtonian fluid in a porous medium (see for example \cite{kula}).
Porzio and Vespri \cite{PV} and Ivanov \cite{Ivanov1991} independently proved that the weak solutions to \eqref{prototype} are H\"older continuous. The Harnack
inequality of weak solutions to doubly degenerate parabolic equations has been established by Vespri \cite{v} and the case of equations with general quasi-linear structure was treated by Fornaro and Sosio \cite{FS}.
Motivated by these work, we are interested in finding the weak Harnack inequalities for positive, weak supersolutions to this kind of parabolic equations.

The aim of this paper is to establish both local and global Harnack estimates for positive, weak supersolutions to \eqref{prototype} in the range
$p+m>3$ and $p>2$. Our proof is in the spirit of \cite[chapter 3]{DGV} which reduce
the proof to the consideration of hot alternative and cold alternative.
Our first goal is to prove the Caccioppoli estimates. The treatment of the weak supersolutions is different from weak solutions, since the test function should be non-negative. Unlike the argument in \cite{FS}, it is not convenient to use Steklov average in the context of supersolutions. Instead, we use the time mollification introduced by Naumann \cite{N}. This kind of the time mollification was also used by Kinnunen and Lindqvist \cite{KL2}-\cite{KL3} to establish the priori estimates of supersolutions.
The major difficulty in the proof of Caccioppoli estimates for the cold alternative stems from the fact that any supersolution to \eqref{prototype} plus a constant may not be a supersolution anymore. We have to use the dampening function introduced by Lehtel\"a \cite{Le} to construct a suitable test function in the proof. This idea has also been used by Ivert, Marola and Masson \cite{IMM} in a different context. Subsequently, we establish a result concerning expansion of positivity and iterate the expansion of positivity to obtain
the hot alternative proposition. Furthermore, we have to work with the cold alternative invented by Kuusi \cite[section 5]{Kuusi}; see also \cite[section 4]{Le}.

An outline of this paper is as follows. We set up notations and state the main results in Section \ref{preliminary}. In Section \ref{parabolic sobolev embedding theorem}, we prove
some parabolic Sobolev inequalities which will be used in Section \ref{Cold alternative}. Subsequently, Section \ref{Caccioppoli estimates} establishes the Caccioppoli estimates, which play a crucial role in the remainder of the proof. Section \ref{Expansion of Positivity} contains a discussion of the expansion of positivity, while in Section \ref{Hot alternative} we
prove the desired estimate in the hot alternative. In Section \ref{Cold alternative}, we use Moser's iterative method to obtain a lower bound for the supersolution in the cold alternative. Finally, in Section \ref{main results}, we finish the proof of the local and global weak Harnack inequalities.

\section{Statement of the main Results}\label{preliminary}

Throughout the paper, $E$ will denote a bounded domain in $\mathbb{R}^N$ and $\partial E$ stand for the boundary of $E$. For $T>0$, let $E_T$ be the cylindrical domain $E\times (0,T]$.
Points in $\mathbb{R}^{N+1}$ will be denoted by $z=(x,t)$ where $x\in \mathbb{R}^N$ and $t\in\mathbb{R}$.

For $f\in C^1(E_T)$, we denote by $Df$ the differentiation with respect to the space variables, while $\partial_tf$ stands for the time derivative.
The spaces $L^p(E)$ and $W^{1,p}(E)$ are the locally Lebesgue and Sobolev spaces. A function $f\in L^p_{\mathrm{loc}}(E)$ if $f\in L^p(K)$ for all compact subset $K\subset E$.
Similarly, the function $f\in W^{1,p}_{\mathrm{loc}}(E)$ if $f\in  W^{1,p}(K)$ for all compact subset $K\subset E$.

Let $k$ be any real number and for a function $v\in W^{1,p}(E)$ the truncations are defined by
\begin{equation*}\begin{split}&(v-k)_+=\max\{v-k;0\}\\
&(v-k)_-=\max\{-(v-k);0\}.\end{split}\end{equation*}

For $\rho>0$ and $y\in \mathbb{R}^N$, denote by $K_\rho(y)$ the cube centered at $y$, with sides parallel to the coordinate axes, and with sides of length $\rho$.
If $y$ is the origin, we abbreviate $K_\rho(y)$ to $K_\rho$. For $A\subset \mathbb{R}^{N+1}$ we denote by $|A|$ the Lebesgue measure of $A$. We define $\chi_A$ be the characteristic function of $A$.

Let $T_1$, $T_2\in \mathbb{R}$ and $T_1<T_2$. Consider quasi-linear parabolic equations of the form
\begin{equation}\label{DDP}\frac{\partial u}{\partial t} - \operatorname{div} A(z,u,Du)=0\quad\text{weakly in}\quad E\times(T_1,T_2),\end{equation}
where the function $A:E\times(T_1,T_2)\times \mathbb{R}\times\mathbb{R}^{N}\to\mathbb{R}^N$ are assumed to be measurable and subject to the structure conditions
\begin{equation}\label{structure def}
	\begin{cases}
	A(x,t,u,Du)\cdot Du\geq C_0|u|^{m-1}|Du|^p,\\
	|A(x,t,u,Du)|\leq C_1|u|^{m-1}|Du|^{p-1},
	\end{cases}
\end{equation}
for almost all $(x,t)\in E\times(T_1,T_2)$, with
\begin{equation}\label{mp}
m+p>3,\quad
	p>2
\end{equation}
and $C_0$, $C_1$ positive constants. We now give the definition of weak supersolutions to doubly degenerate parabolic equations.
\begin{definition} A function $u:E\times(T_1,T_2]\to\mathbb{R}$ is said to be a local weak supersolution of \eqref{DDP}-\eqref{mp} if
\begin{equation}\begin{split}\label{regularity}u\in &C_{loc}(T_1,T_2;L_{loc}^{2}(E)),\quad u^{\alpha}\in L_{loc}^{p}(T_1,T_2;W_{loc}^{1,p}(E))\quad\mathrm{where}\quad\alpha=\tfrac{m+p-2}{p-1}\end{split}\end{equation}
and the inequality
\begin{equation}\label{weak form deff}
-\int_{T_1}^{T_2}\int_Eu\frac{\partial \varphi }{\partial t}dxdt+ \int_{T_1}^{T_2}\int_EA(x,t,u,Du)\cdot D\varphi\ dxdt\geq 0\end{equation}
holds for any non-negative test function $\varphi\in C_0^\infty(E\times(T_1,T_2))$.\end{definition}
We remark that the gradient of $u$ in \eqref{structure def} is interpreted as
$Du=\alpha^{-1} u^{1-\alpha}Du^\alpha.$ Throughout this paper, we assume that weak supersolution $u$ is \emph{positive} almost everywhere, i.e., $|[u=0]|=0$.

The following theorems are our main results. We prove that for any positive, weak supersolutions to \eqref{DDP}-\eqref{mp} defined above, local weak Harnack estimate takes the following form.
\begin{theorem}\label{main result1} Let $u$ be a positive, weak super-solution to \eqref{DDP}-\eqref{mp} in an open set which compactly contains $E_T$. For $y\in E$, let $\rho>0$ be so small that $K_{32\rho}(y)\subset E$.
There exist positive constants $c$ and $\gamma$, depending
only upon $m$, $p$, $N$, $C_0$ and $C_1$, such that for almost every $s\in (0,T)$,
\begin{equation}\label{weak harnack local}\fint_{K_\rho(y)}u(x,s)dx\leq c\left(\frac{\rho^p}{T-s}\right)^{\frac{1}{m+p-3}}+
\gamma \essinf_{Q_{\rho,\theta}}u,\end{equation}
where
$Q_{\rho,\theta}=K_{4\rho}(y)\times\left(s+\tfrac{1}{2}\theta \rho^p, s+\theta \rho^p\right]$
and
\begin{equation}\label{weak harnack local theta}\theta=\min\left\{c\frac{T-s}{\rho^p},\quad \left(\fint_{K_\rho(y)}u(x,s)dx\right)^{3-m-p}\right\}.\end{equation}
\end{theorem}
In the global case, Theorem \ref{main result1} can be sharpened to the following result.
\begin{theorem}\label{main result2} Let $u$ be a positive, weak super-solution to \eqref{DDP}-\eqref{mp} in $\mathbb{R}^N\times(0,T]$. There exists a positive constant $\gamma$, depending
only upon $m$, $p$, $N$, $C_0$ and $C_1$, such that for all $(y,s)\in \mathbb{R}^N\times(0,T)$, $\rho>0$ and $T^\prime>0$ such that $s+T^\prime<T$, there holds
\begin{equation}\label{weak harnack global}\fint_{K_\rho(y)}u(x,s)dx\leq \gamma\left(\frac{\rho^p}{T^\prime}\right)^{\frac{1}{m+p-3}}+
\gamma \left(\frac{T^\prime}{\rho^p}\right)^{\frac{N}{p}}\large( \essinf_{\tilde Q_{\rho,T^\prime}}u\large)^{\frac{\lambda}{p}},\end{equation}
where $\lambda=N(p+m-3)+p$ and
$\widetilde{Q}_{\rho,T^\prime}=K_{4\rho}(y)\times \left(s+\tfrac{1}{2}T^\prime,s+T^\prime\right].$
\end{theorem}
Note that in the case $m=1$, doubly degenerate parabolic equations \eqref{DDP}-\eqref{mp} reduce to the evolutionary $p$-Laplace equations. In this case, Theorem \ref{main result1} and Theorem \ref{main result2} were proved
in \cite{Kuusi} and \cite[chapter 3]{DGV}. Our proof will focus on two cases $p>2$, $m>1$ and $0<m<1$, $m+p>3$.

\section{Parabolic Sobolev inequalities}\label{parabolic sobolev embedding theorem}
In this section, we collect some parabolic Sobolev inequalities which will be used in Section \ref{Expansion of Positivity}-\ref{Cold alternative}. Here and subsequently,
the function space $W_0^{1,p}(E)$ is defined as the set of functions $v\in W^{1,p}(E)$ whose trace on $\partial E$ is zero.
In the following, we present some of the results obtained in \cite[chapter I]{Di}.

\begin{proposition}\cite[Chapter I, Proposition 3.1]{Di}\label{parabolic sobolev} Let $p\geq1$ and $r\geq 1$. There exists a constant $\gamma$ depending only upon $N$, $p$ and $r$ such that for every
$v\in L^\infty(0,T;L^r(E))\cap L^p(0,T;W_0^{1,p}(E))$, there holds
\begin{equation}\begin{split}\label{Sobolev formula}\iint_{E_T}|v(x,t)|^qdz\leq &\gamma \iint_{E_T}|Dv(x,t)|^pdz
\left(\esssup_{0<t<T}\int_E|v(x,t)|^rdx\right)^{\frac{p}{N}}\end{split}\end{equation}
where $q=p\frac{N+r}{N}$.
\end{proposition}
\begin{lemma}\cite[Chapter I, Theorem 2.1]{Di}\label{gargliardo-nirenberg} Let $v\in W_0^{1,p}(E)$ with $p\geq1$. For $s\geq1$ there exists a constant $C$ depending only upon $N$, $p$ and $s$ such that
\begin{equation}\label{GNformula}\left(\int_{E}|v|^qdx\right)^{\frac{1}{q}}\leq C\left(\int_{E}|Dv|^pdx\right)^{\frac{\alpha}{p}}\left(\int_{E}|v|^sdx\right)^{\frac{1-\alpha}{s}}\end{equation}
where $\alpha\in[0,1]$, $p\geq1$, $q\geq1$ and
\begin{equation}\label{index}\alpha=\left(\frac{1}{s}-\frac{1}{q}\right)\left(\frac{1}{N}-\frac{1}{p}+\frac{1}{s}\right)^{-1}.\end{equation}
\end{lemma}
We remark that the estimate \eqref{GNformula} is also known as  Gagliardo-Nirenberg interpolation inequality and this will be used in the proof of the following result.
\begin{proposition}\label{m<1Sobolev} Let $p\geq1$ and $r>0$. Set
$$Q_1=K_1\times\Lambda_1=\left[-\tfrac{1}{2},\tfrac{1}{2}\right]^N\times[0,1].$$
Then there exists a constant $\gamma$ depending only upon $N$, $p$ and $r$ such that, for any $v\in L^\infty(0,1;L^{r}(K_1))\cap L^p(0,1;W_0^{1,p}(K_1))$,
the following holds:

In the case $p<N$, there holds
\begin{equation}\begin{split}\label{p<N r<1}\iint_{Q_1}|v(x,t)|^{p\frac{N+r}{N}}dz\leq &\gamma \iint_{Q_1}|Dv(x,t)|^pdz
\left(\esssup_{0<t<1}\int_{K_1}|v(x,t)|^rdx\right)^{\frac{p}{N}}.\end{split}\end{equation}
In the case $p>N$, the estimate
\begin{equation}\begin{split}\label{p>N r<1}\iint_{Q_1}|v(x,t)|^{p+\frac{r}{q}}dz\leq &\gamma \iint_{Q_1}|Dv(x,t)|^pdz
\left(\esssup_{0<t<1}\int_{K_1}|v(x,t)|^{r}dx\right)^{\frac{1}{q}}\end{split}\end{equation}
holds for any $q>1$.
In the case $p=N$, there holds
\begin{equation}\begin{split}\label{p=N r<1}\iint_{Q_1}|v(x,t)|^{p+\frac{r}{2}}dz\leq &\gamma \iint_{Q_1}|Dv(x,t)|^pdz
\left(\esssup_{0<t<1}\int_{K_1}|v(x,t)|^{r}dx\right)^{\frac{1}{2}}.\end{split}\end{equation}
\end{proposition}
\begin{proof} Our proof strategy is to use H\"older's inequality together with Sobolev's inequality.
In the case $p<N$, the proof is due to DiBenedetto \cite[page 8]{Di} and we include the proof here for the sake of completeness. Applying Sobolev embedding theorem slicewise to $v(\cdot,t)$, we obtain
\begin{equation*}\begin{split}\left(\int_{K_1}|v(\cdot,t)|^{\frac{pN}{N-p}}dx\right)^{\frac{N-p}{N}}
\leq C(N,p)\int_{K_1}|Dv(\cdot,t)|^pdx.\end{split}\end{equation*}
This yields
\begin{equation*}\begin{split}\iint_{Q_1}|v|^{p\frac{N+r}{N}}dz&\leq \int_0^1\left(\int_{K_1}|v(x,t)|^{\frac{pN}{N-p}}dx\right)^{\frac{N-p}{N}}
\left(\int_{K_1}|v(x,t)|^{r}dx\right)^{\frac{p}{N}}dt
\\&\leq C(N,p)\iint_{Q_1}|Dv|^pdz
\left(\esssup_{0<t<1}\int_{K_1}|v(x,t)|^rdx\right)^{\frac{p}{N}},\end{split}\end{equation*}
and \eqref{p<N r<1} is proved.
We now turn to the case $p>N$. In this case, the proof is due to Kuusi \cite{Kuusipersonal}. We use the Sobolev's inequality slicewise on $K_1$ and obtain
$$\sup_{x\in K_1}|v(\cdot,t)|\leq C(N,p)\left(\int_{K_1}|Dv(\cdot,t)|^pdx\right)^{\frac{1}{p}}.$$
Then, for any $\bar q>1$, there holds
\begin{equation*}\begin{split}\label{sobolev q>1}\left(\int_{K_1}|v(\cdot,t)|^{p\bar q}dx\right)^{\frac{1}{\bar q}}\leq \left(\sup_{x\in K_1}|v(x,t)|\right)^p
\leq C(N,p)\int_{K_1}|Dv(\cdot,t)|^pdx.\end{split}\end{equation*}
Choosing $\bar q=q/(q-1)$, we conclude from H\"older's inequality that
\begin{equation*}\begin{split}\iint_{Q_1}|v|^{p+\frac{r}{q}}dz&\leq \int_0^1\left(\int_{K_1}|v(x,t)|^{p\bar q}dx\right)^{\frac{1}{\bar q}}\left(\int_{K_1}|v(x,t)|^{r}dx\right)^{\frac{1}{q}}dt
\\&\leq C\iint_{Q_1}|Dv|^pdz\left(\esssup_{0<t<1}\int_{K_1}|v(\cdot,t)|^{r}dx\right)^{\frac{1}{q}},\end{split}\end{equation*}
and \eqref{p>N r<1} follows.
Finally, we come to the case $p=N$ which is the limiting case of Sobolev's inequality. We apply Lemma \ref{gargliardo-nirenberg} \eqref{GNformula} with $p=s=N$, $q=2N$ and $\alpha=\frac{1}{2}$ to obtain
\begin{equation}\label{GNp=N}\int_{\Omega}|v(\cdot,t)|^{2N}dx\leq C\int_{\Omega}|Dv(\cdot,t)|^{N}dx\int_{\Omega}|v(\cdot,t)|^Ndx.\end{equation}
Note that $v(\cdot,t)\in W_0^{1,p}(K_1)$ for almost every $t\in [0,1]$. Applying Friedrichs' inequality slicewise to $v(\cdot,t)$, we get
\begin{equation}\label{Friedrichs}\int_{\Omega}|v(\cdot,t)|^Ndx\leq C(N)\int_{\Omega}|Dv(\cdot,t)|^Ndx.\end{equation}
We plug \eqref{Friedrichs} back into \eqref{GNp=N}, there holds
\begin{equation*}\int_{\Omega}|v(\cdot,t)|^{2N}dx\leq C\left(\int_{\Omega}|Dv(\cdot,t)|^{N}dx\right)^2.\end{equation*}
By H\"older's inequality, we deduce
\begin{equation*}\begin{split}\iint_{Q_1}|v|^{N+\frac{r}{2}}dxdt&\leq \int_0^1\left(\int_{K_1}|v(x,t)|^{2N}dx\right)^{\frac{1}{2}}\left(\int_{K_1}|v(x,t)|^{r}dx\right)^{\frac{1}{2}}dt
\\&\leq C\iint_{Q_1}|Dv|^Ndz\left(\esssup_{0<t<1}\int_{K_1}|v(\cdot,t)|^{r}dx\right)^{\frac{1}{2}},\end{split}\end{equation*}
which completes the proof.
\end{proof}
\begin{remark}
\end{remark}
\begin{enumerate}
\item Compared with Proposition \ref{parabolic sobolev}, the condition for $r$ in Proposition \ref{m<1Sobolev} is relaxed and $0<r<1$ is admissible. For the applications,
we shall use Proposition \ref{parabolic sobolev} together with De Giorgi's estimates in Section \ref{Expansion of Positivity}-\ref{Hot alternative}. Proposition \ref{m<1Sobolev} will be used in the proof of Lemma \ref{moser integral} and Lemma \ref{uniform bound} in Section \ref{Cold alternative}.
\item After finishing this paper, we became aware of \cite{ACCN}, which proves Proposition \ref{m<1Sobolev} in a more general Sub-Riemannian homogeneous spaces. However, our approach make the argument simpler and easier.
\end{enumerate}

\section{Caccioppoli Estimates}\label{Caccioppoli estimates}

Throughout this section, $Q$ denotes the open subset of $E$. We shall use parabolic cylinders of the form $Q_T=Q\times[0,T]$.
Let $\partial_PQ_T=\partial Q\times[0,T]\cup Q\times\{0\}$ denote the parabolic boundary
of $Q_T$. For the fixed $t_1$, $t_2\in\mathbb{R}$ with $t_1<t_2$,
if the test function $\Phi(x,t)$ is required to vanish only on the lateral boundary $\partial Q\times[t_1,t_2]$, then the boundary terms
\begin{equation}\label{boundary t1}\int_Qu(x,t_1)\Phi(x,t_1)dx:=\lim_{\nu\to0}\frac{1}{\nu}\int_{t_1}^{t_1+\nu}\int_Qu(x,t)\Phi(x,t)dxdt\end{equation}and
\begin{equation}\label{boundary t2}\int_Qu(x,t_2)\Phi(x,t_2)dx:=\lim_{\nu\to0}\frac{1}{\nu}\int_{t_2-\nu}^{t_2}\int_Qu(x,t)\Phi(x,t)dxdt\end{equation}
have to be included. Next we shall derive an inequality involving the boundary terms which is an equivalent form of \eqref{weak form deff}.
Let $\nu<\frac{1}{4}(t_2-t_1)$ be a fixed positive number. We introduce the Lipschitz function $\theta_\nu(t)$ by\begin{equation}\label{theta time}
	\theta_\nu(t)=\begin{cases}
	0,&t<t_1,\\
	\frac{1}{\nu}(t-t_1),&t_1\leq t<t_1+\nu,\\ 1,&t_1+\nu\leq t<t_2-\nu,\\ \frac{1}{\nu}(t_2-t),&t_2-\nu\leq t<t_2,\\ 0,&t\geq t_2.
	\end{cases}
\end{equation}
Let $u$ be a weak supersolution to \eqref{DDP}-\eqref{mp} in an open set which compactly contains $Q_{t_1,t_2}=Q\times[t_1,t_2]$ and $\Phi$ be the smooth function mentioned above, we choose the test function $\varphi=\Phi(x,t)\theta_\nu(t)$ in \eqref{weak form deff} and pass to the limit $\nu\downarrow0$. Then we obtain from \eqref{boundary t1} and \eqref{boundary t2} that
\begin{equation}\begin{split}\label{weak form 1}
-\int_{t_1}^{t_2}\int_{Q}&u\frac{\partial \Phi }{\partial t}dxdt+ \int_{t_1}^{t_2}\int_{Q}A(x,t,u,Du)\cdot D\Phi dxdt
\\&+\int_Qu(x,t_2)\Phi(x,t_2)dx-\int_Qu(x,t_1)\Phi(x,t_1)dx\geq 0.\end{split}\end{equation}
We remark that \eqref{weak form 1} is an equivalent form of \eqref{weak form deff} and this will be used in the proof of Proposition \ref{averaged proposition}. Specifically, in the case when $\Phi$ is independent of $t$, we have
\begin{equation}\begin{split}\label{weak form 22}
\int_Qu(x,t_2)\Phi(x)dx\geq&\int_Qu(x,t_1)\Phi(x)dx-\iint_{Q_{t_1,t_2}}A(x,t,u,Du)\cdot D\Phi dz.\end{split}\end{equation}
This inequality is only needed in the proof of Proposition \ref{cold alt pro} in Section \ref{Cold alternative}.

In order to deal with the possible lack of differentiability in
time of weak supersolutions, the following time mollification of functions $v$ has
proved to be useful. We now define
$$v^*(x,t)=\frac{1}{\sigma}\int_0^te^{(s-t)/\sigma}v(x,s)ds.$$
Some elementary properties of this mollification are listed in the following lemma (see for example \cite[Lemma 2.9]{KL1}, \cite[Appendix B]{BDM}, \cite[page 36]{N}).
\begin{lemma}  \label{averaged function property}
\begin{enumerate} Let $q\geq1$.
\item If $v\in L^q(Q_T)$, then $$\|v^*\|_{L^q(Q_T)}\leq \|v\|_{L^q(Q_T)}$$
and
\begin{equation}\label{averaged function 1}\frac{\partial v^*}{\partial t}=\frac{1}{\sigma}(v-v^*)\in L^q(Q_T).\end{equation}
Moreover,  $v^*\to v$ in $L^q(Q_T)$ as $\sigma\downarrow0$.
\item If, in addition, $Dv\in L^q(Q_T)$, then $D(v^*)=(Dv)^*$ and $Dv^*\to Dv$ in $L^q(Q_T)$ as $\sigma\downarrow0$.
\item If $v\in C(\overline Q_T)$, then
$$v^*(x,t)+e^{-t/\sigma}v(x,0)\to v(x,t)$$
uniformly in $\overline Q_T$ as $\sigma\downarrow0$.
\end{enumerate}
\end{lemma}
The following result is due to Kinnunen and Lindqvist \cite{KL2}-\cite{KL3}, which is crucial in the proof of Caccioppoli estimates. For completeness sake we include the proof here.
\begin{proposition}\label{averaged proposition}Let $u$ be a weak supersolution to \eqref{DDP}-\eqref{mp} in an open set which compactly contains $Q_T$. If $\sigma<T/4$, then
\begin{equation}\begin{split}\label{averaged equation}\iint_{Q_T}\left(\left[A(x,t,u,Du)\right]^*\cdot D\Phi+\Phi\frac{\partial u^*}{\partial t} \right)dz
\geq0\end{split}\end{equation}
valid for all non-negative function $\Phi\in L^p(0,T;W_0^{1,p}(Q))$.
\end{proposition}
Before we can prove the Proposition \ref{averaged proposition}, we need an approximation lemma.
\begin{lemma}\label{density} Let $v\in L^p(0,T;W_0^{1,p}(Q))$. For any $\epsilon>0$, there exists a smooth function $\tilde v$ vanishing on the lateral boundary $\partial Q\times[0,T]$ such that
$$\|v-\tilde v\|_{L^p(Q_T)}+\|Dv-D\tilde v\|_{L^p(Q_T)}<\epsilon.$$\end{lemma}
\begin{proof} We apply Proposition 6.29 of \cite[page 200]{hunter} with $X=W_0^{1,p}(Q)$. There exist $\theta_i\in C_0^\infty(0,T)$ and $w_i\in W_0^{1,p}(Q)$ with $i=1,2,\cdots,n$,
such that
$$\large\|\sum_{i=1}^nw_i\theta_i-v\large\|_{L^p(Q_T)}+\large\|\sum_{i=1}^n\theta_i Dw_i-Dv\large\|_{L^p(Q_T)}<\frac{\epsilon}{2}.$$
Let $\psi_i\in C_0^\infty(Q)$ be such that
$$\|w_i-\psi_i\|_{L^p(Q)}+\|Dw_i-D\psi_i\|_{L^p(Q)}<\frac{\epsilon}{2^i\max\{1,\|\theta_i\|_{L^p[0,T]}\}}$$
for any $i=1,2,\cdots,n$. If we choose
$\tilde v=\sum_{i=1}^n\psi_i\theta_i$, then the lemma follows by triangle inequalities.
\end{proof}

\begin{proof}[Proof of Proposition \ref{averaged proposition}] Let us assume initially that $\Phi$ is a non-negative smooth function vanishing on the lateral boundary $\partial Q\times[0,T]$. For a fixed $s\in\left(0,T\right)$, we apply \eqref{weak form 1} with $t_1=0$, $t_2=T-s$ and choose the test function $\Phi(x,t+s)$ to obtain
\begin{equation*}\begin{split}
\int_{s}^{T}\int_{Q}&-u(x,t-s)\frac{\partial \Phi }{\partial t}+ A\left(x,t-s,u(x,t-s),Du(x,t-s)\right)\cdot D \Phi dxdt
\\&+\int_Qu(x,T-s)\Phi(x,T)dx-\int_Qu(x,0)\Phi(x,s)dx\geq 0.\end{split}\end{equation*}
Multiplying both sides by $\sigma^{-1}e^{-s/\sigma}$ and integrating over $[0,T]$ with respect to $s$, we get
\begin{equation*}\begin{split}\int_{0}^{T}\int_{Q}&\left(\left[A(x,t,u,Du)\right]^*\cdot D\Phi-u^*\frac{\partial \Phi}{\partial t} \right)dxdt
+\int_Qu^*(x,T)\Phi(x,T)dx\\&\geq \int_Qu(x,0)\left(\frac{1}{\sigma}\int_0^{T}\Phi(x,s)e^{-s/\sigma}ds\right)dx.\end{split}\end{equation*}
Using integration by parts and noting that $u^*(x,0)=0$, we have
\begin{equation*}\begin{split}\int_{0}^{T}\int_{Q}&\left(\left[A(x,t,u,Du)\right]^*\cdot D\Phi+\Phi\frac{\partial u^*}{\partial t} \right)dxdt
\geq \int_Qu(x,0)\left(\frac{1}{\sigma}\int_0^{T}\Phi(x,s)e^{-s/\sigma}ds\right)dx.\end{split}\end{equation*}
It follows that
\eqref{averaged equation}
holds for any non-negative smooth functions $\Phi$ which vanishes on the lateral boundary $\partial Q\times[0,T]$. Finally, we turn our attention to the case when
$\Phi\in L^p(0,T;W_0^{1,p}(Q))$. According to Lemma \ref{density}, there exists a sequence of smooth functions $\{\phi_j\}_{j=1}^\infty$ which satisfy
\begin{equation*}\begin{split}\label{averaged proof 11}\int_{0}^{T}\int_{Q}&\left(\left[A(x,t,u,Du)\right]^*\cdot D\Phi_j+\Phi_j\frac{\partial u^*}{\partial t} \right)dxdt
\geq 0\end{split}\end{equation*} and
$\phi_j\to\phi$ in $L^p(0,T;W_0^{1,p}(Q))$ as $j\to \infty$. We use Lemma \ref{averaged function property} (1) with $q=p/(p-1)$ and H\"older's inequality to find that
\begin{equation*}\begin{split}\label{averaged proof}\int_{0}^{T}\int_{Q}\left[A(x,t,u,Du)\right]^*\cdot D\Phi_j dxdt\to \int_{0}^{T}\int_{Q}\left[A(x,t,u,Du)\right]^*\cdot D\Phi dxdt\end{split}\end{equation*}
as $j\to \infty$. As for the term involving time derivative, we use \eqref{averaged function 1} with $q=p$ to infer that $\partial u^*/\partial t\in L^p(Q_T)$.
Let $p^\prime=p/(p-1)$. Since $p>2$, we have $\phi_j\to\phi$ in $L^{p^\prime}(Q_T)$ as $j\to \infty$. By H\"older's inequality, we conclude that
\begin{equation*}\begin{split}\label{averaged proof 1}\int_{0}^{T}\int_{Q}\Phi_j\frac{\partial u^*}{\partial t}dxdt
\to \int_{0}^{T}\int_{Q}\Phi\frac{\partial u^*}{\partial t}dxdt\end{split}\end{equation*}
as $j\to \infty$, which proves the proposition.
\end{proof}
We are now in a position to study the Caccioppoli estimates. The Caccioppoi inequality stated in next proposition, was first announced by Fornaro and Sosio \cite[(4.7)]{FS}, deals with the case $m\geq1$ and $p>2$. However, there is no proof given for this estimate in \cite{FS}. We shall prove this result in the following.
\begin{proposition}\label{C1} Let $m> 1$ and $p>2$. Let $u$ be a positive, weak supersolution to \eqref{DDP}-\eqref{mp} in an open set which compactly contains $Q_T$. There exists a positive constant $\gamma$ depending only upon $m$, $p$, $C_0$ and $C_1$,
such that for every cylinder $Q_{t_1,t_2}=Q\times(t_1,t_2)\subset Q_T$, every $k>0$ and every piecewise smooth, nonnegative cutoff function $\zeta$ vanishing on $\partial Q\times[t_1,t_2]$,
\begin{equation}\begin{split}\label{CC}\tfrac{1}{2}\esssup_{t_1\leq \tau\leq t_2}
\int_{Q}&(u-k)_-^2(\cdot,\tau)\zeta^p(\cdot,\tau)dx+\iint_{Q_{t_1,t_2}}u^{m-1}|D(u-k)_-|^p\zeta^pdz\\ \leq &\tfrac{1}{2}\iint_{Q_{t_1,t_2}}(u-k)_-^2\left|\frac{\partial \zeta^p}{\partial t}\right|dz\\&+\tfrac{1}{2}\int_{Q}(u-k)_-^2(\cdot,t_1)\zeta^p(\cdot,t_1)dx\\&\quad+\gamma \iint_{Q_{t_1,t_2}}u^{m-1}(u-k)_-^p|D\zeta|^pdz.\end{split}\end{equation}
\end{proposition}
\begin{proof}
Since the supersolution multiplied by a cut-off function cannot be used as testing function in the formula \eqref{averaged equation}, we have to construct a suitable testing function. Let $\alpha=\frac{m+p-2}{p-1}$ and observe that $\alpha>1$. For a fixed $\delta>0$, we construct a function $G_\delta(v)$ by
\begin{equation*}
	G_\delta(v)=\begin{cases}
	v^{\frac{1}{\alpha}},&\mathrm{if}\quad v>\delta,\\ \delta^{\frac{1}{\alpha}}+\frac{1}{\alpha}\delta^{\frac{1}{\alpha}-1}(u-\delta),&\mathrm{if}\quad v\leq\delta.
	\end{cases}
\end{equation*}
It can be easily seen that $G_\delta(v)$ is a Lipschitz function with respect to $v$.
Moreover, for a fixed $\lambda$ with $0<\lambda<\tfrac{1}{10}k$, we introduce an auxiliary function
\begin{equation*}
	F_\lambda(s)=\begin{cases}
	0,&\mathrm{if}\quad s\geq k,\\
	k-s,&\mathrm{if}\quad \lambda<s< k,\\ k-\lambda,&\mathrm{if}\quad s\leq\lambda.
	\end{cases}
\end{equation*}
Alternatively,
$F_\lambda(s)=\left[(s-\lambda)_++\lambda-k\right]_-$ for all $s>0$.
It is easy to check that $F_\lambda(s)$ is a Lipschitz function as well. At this stage, we define
$$\tilde F(v)=F_\lambda\large[G_{\tfrac{\lambda^\alpha}{2}}(v)\large].$$
We conclude from \eqref{regularity} and \cite[Theorem 2.1.11]{Z} that $\tilde F(u^\alpha)(\cdot,t)\in W^{1,p}(Q)$ for almost every $t\in(0,T)$ and there holds
\begin{equation}\begin{split}\label{truncation}\frac{\partial}{\partial x_i} \tilde F(u^\alpha)
&=-\chi_{[\lambda<G_{\tfrac{1}{2}\lambda^\alpha}(u^\alpha)<k]}\frac{\partial}{\partial x_i}G_{\tfrac{\lambda^\alpha}{2}}(u^\alpha)
=-\chi_{[\lambda<u<k]}G_{\tfrac{\lambda^\alpha}{2}}^\prime(u^\alpha)\frac{\partial u^\alpha}{\partial x_i}\\&=-\alpha^{-1}\chi_{[\lambda<u<k]}u^{1-\alpha}\frac{\partial u^\alpha}{\partial x_i}\end{split}\end{equation}
where $i=1,\cdots,N $.
Hence $ \tilde F(u^\alpha)\in L^p(0,T;W^{1,p}(Q))$ and $D\tilde F(u^\alpha)=-\chi_{[\lambda<u<k]}Du$.

We are now in a position to construct the testing function.
Let $\tau\in (t_1,t_2)$ be a fixed instant and $Q_\tau=Q\times(t_1,\tau)$. Set $\theta_\nu=\theta_\nu(t)$ be as in \eqref{theta time} with $t_2$ replaced by $\tau$. In the inequality \eqref{averaged equation} we choose $\Phi=\tilde F(u^\alpha)\zeta^p\theta_\nu$ as a testing function. Since $\tilde F(u^\alpha)=F_\lambda(u)$, we conclude that
\begin{equation}\begin{split}\label{CC1}\iint_{Q_\tau}F_\lambda(u)\zeta^p\theta_\nu\frac{\partial u^*}{\partial t} dz
+\iint_{Q_\tau}&\left[A(x,t,u,Du)\right]^*\cdot D\left[F_\lambda(u)\zeta^p\theta_\nu\right]dz
\geq0.\end{split}\end{equation}
From \eqref{averaged function 1}, we have $\partial u^*/\partial t=(u-u^*)/\sigma$. We decompose the first integral as follows
\begin{equation*}\begin{split}\iint_{Q_\tau}&F_\lambda(u)\zeta^p\theta_\nu\frac{\partial u^*}{\partial t} dz
=\iint_{Q_\tau}F_\lambda(u^*)\zeta^p\theta_\nu\frac{\partial u^*}{\partial t} dz
+\iint_{Q_\tau}\zeta^p\theta_\nu\left[F_\lambda(u)-F_\lambda(u^*)\right]\frac{u-u^*}{\sigma} dz.\end{split}\end{equation*}
Since $F_\lambda(s)$ is decreasing with respect to $s$, then
\begin{equation*}\iint_{Q_\tau}\zeta^p\theta_\nu\left[F_\lambda(u)-F_\lambda(u^*)\right]\frac{u-u^*}{\sigma} dz\leq0,\end{equation*}
which implies that
\begin{equation}\begin{split}\label{CC2}\iint_{Q_\tau}&F_\lambda(u)\zeta^p\theta_\nu\frac{\partial u^*}{\partial t} dz
\leq \iint_{Q_\tau}F_\lambda(u^*)\zeta^p\theta_\nu\frac{\partial u^*}{\partial t} dz.\end{split}\end{equation}
The task is now to estimate the right hand side of \eqref{CC2}. We set $f_\lambda(s)=\int_{s}^kF_\lambda(r)dr$ and observe that
\begin{equation}\label{fs}
	f_\lambda(s)=\begin{cases}
	\frac{1}{2}(k-\lambda)^2+(\lambda-s)(k-\lambda),&0<s\leq\lambda,\\
	\frac{1}{2}(k-s)^2,&\lambda<s\leq k,\\ 0,&s> k.
	\end{cases}
\end{equation}
Moreover, we see that
$$F_\lambda(u^*)\frac{\partial u^*}{\partial t}=-\frac{\partial f_\lambda (u^*)}{\partial t}.$$
Using integration by parts and noting that $\theta_\nu(t_1)=\theta_\nu(\tau)=0$, we obtain
\begin{equation}\begin{split}\label{CC3}\iint_{Q_\tau}F_\lambda(u^*)\zeta^p\theta_\nu\frac{\partial u^*}{\partial t} dz
=\iint_{Q_\tau}f_\lambda(u^*)\zeta^p\frac{\partial \theta_\nu}{\partial t} dz+\iint_{Q_\tau}f_\lambda(u^*)\theta_\nu\frac{\partial \zeta^p}{\partial t} dz.\end{split}\end{equation}
By substituting \eqref{CC3} into \eqref{CC2}, we obtain from \eqref{CC1} that
\begin{equation}\begin{split}\label{CC4}
\iint_{Q_\tau}&\left[A(x,t,u,Du)\right]^*\cdot D\left[F_\lambda(u)\zeta^p\right]\theta_\nu dz\\&
\geq -\iint_{Q_\tau}f_\lambda(u^*)\zeta^p\frac{\partial \theta_\nu}{\partial t} dz-\iint_{Q_\tau}f_\lambda(u^*)\theta_\nu\frac{\partial \zeta^p}{\partial t} dz.\end{split}\end{equation}
Taking \eqref{fs} into consideration, we check at once that $f_\lambda(s)$ is a Lipschitz function with respect to $s$. Therefore we can pass to the limit $\sigma\downarrow 0$ in \eqref{CC4}.
We first let $\sigma\downarrow 0$ and then $\nu\downarrow0$ in \eqref{CC4}, there holds
\begin{equation}\begin{split}\label{CC5}\iint_{Q_\tau}&A(x,t,u,Du)\cdot D\left(F_\lambda(u)\zeta^p\right)dz
\\ &+\iint_{Q_\tau}f_\lambda(u)\frac{\partial \zeta^p}{\partial t}dz
+\int_{Q}f_\lambda(u)(\cdot,t_1)\zeta^p(\cdot,t_1)dx
\\&\geq \int_{Q}f_\lambda(u)(\cdot,\tau)\zeta^p(\cdot,\tau)dx.\end{split}\end{equation}
At this stage, we decompose the first term on the left hand side,
\begin{equation}\begin{split}\label{CC6}\iint_{Q_\tau}&A(x,t,u,Du)\cdot D\left(F_\lambda(u)\zeta^p\right)dz\\&=\iint_{Q_\tau}F_\lambda^\prime(u) \zeta^p A(x,t,u,Du)\cdot Du dz\\&\quad+\iint_{Q_\tau}p\zeta^{p-1}F_\lambda(u)A(x,t,u,Du)\cdot D\zeta dz\\&=:T_1(\lambda)+T_2(\lambda).\end{split}\end{equation}
To estimate $T_1(\lambda)$, we use \eqref{structure def} to obtain
\begin{equation}\label{CC7}T_1(\lambda)\leq -C_0\iint_{Q_\tau\cap[\lambda<u<k]}u^{m-1}|Du|^p\zeta^pdz=:-C_0S(\lambda),\end{equation}
with the obvious meaning of $S(\lambda)$. To proceed further, we check that $S(\lambda)<+\infty$ for any $\lambda>0$. This can be seen as follows
$$S(\lambda)= \frac{1}{\alpha^p}\iint_{Q_\tau\cap[\lambda<u<k]}u^{-\frac{m-1}{p-1}}|Du^\alpha|^p\zeta^pdz\leq \frac{\lambda^{-\frac{m-1}{p-1}}}{\alpha^p}\iint_{Q_\tau}|Du^\alpha|^p\zeta^pdz<+\infty,$$
since $Du=\alpha^{-1}u^{1-\alpha}Du^\alpha$, $m>1$ and $p>2$. Now we come to the estimate of $T_2(\lambda)$. Applying Young's inequality and \eqref{structure def}, we obtain
\begin{equation}\begin{split}\label{CC8}T_2(\lambda)&\leq C_1p\alpha^{1-p}k\iint_{Q_\tau\cap[u\leq \lambda]}|Du^\alpha|^{p-1}\zeta^{p-1}|D\zeta|dz\\&+\frac{C_0}{2}\iint_{Q_\tau\cap[\lambda<u<k]}u^{m-1}|Du|^p\zeta^pdz+\frac{2p^{p^\prime}C_1^{p^\prime}}{C_0(p^\prime)^2} \iint_{Q_\tau}u^{m-1}(u-k)_-^p|D\zeta|^pdz.\end{split}\end{equation}
where $p^\prime=p/(p-1)$.
Combining estimates \eqref{CC5}-\eqref{CC8} and taking into account that $S(\lambda)<+\infty$, we absorb $S(\lambda)$ into the left hand side and obtain the estimate
\begin{equation}\begin{split}\label{CC9}\iint_{Q_\tau\cap[u>\lambda]}&u^{m-1}|D(u-k)_-|^p\zeta^pdz+\int_{Q}f_\lambda(u)(\cdot,\tau)\zeta^p(\cdot,\tau)dx\\ \leq &\iint_{Q_\tau}f_\lambda(u)\left|\frac{\partial \zeta^p}{\partial t}\right|dz+\int_{Q}f_\lambda(u)(\cdot,t_1)\zeta^p(\cdot,t_1)dx\\&+\gamma k\iint_{Q_\tau\cap[u\leq \lambda]}|Du^\alpha|^{p-1}\zeta^{p-1}|D\zeta|dz+\gamma \iint_{Q_\tau}u^{m-1}(u-k)_-^p|D\zeta|^pdz,\end{split}\end{equation}
where $\gamma$ depends only on $m$, $p$, $C_0$ and $C_1$. Our task now is to pass to the limit $\lambda\downarrow 0$ in \eqref{CC9}.
We observe from \eqref{fs} that
$f_\lambda(s)$
is decreasing with respect to $\lambda$. By monotone convergence theorem, we deduce
\begin{equation*}\begin{split}&\int_{Q}f_\lambda(u)(\cdot,\tau)\zeta^p(\cdot,\tau)dx\to\frac{1}{2}\int_{Q}(u-k)_-^2(\cdot,\tau)\zeta^p(\cdot,\tau)dx,\\&
\int_{Q}f_\lambda(u)(\cdot,t_1)\zeta^p(\cdot,\tau)dx\to\frac{1}{2}\int_{Q}(u-k)_-^2(\cdot,t_1)\zeta^p(\cdot,\tau)dx,\quad\mathrm{and}
\\&\iint_{Q_\tau}f_\lambda(u)\left|\frac{\partial \zeta^p}{\partial t}\right|dz\to \frac{1}{2}\iint_{Q_\tau}(u-k)_-^2\left|\frac{\partial \zeta^p}{\partial t}\right|dz,
\end{split}\end{equation*}
as $\lambda\downarrow 0$.
Furthermore, we conclude from \eqref{regularity} and Lebesgue dominated theorem that
$$\iint_{Q_\tau\cap[u\leq \lambda]}|Du^\alpha|^{p-1}\zeta^{p-1}|D\zeta|dz\to0\quad\mathrm{as}\quad\lambda\downarrow 0,$$
hence that
$$\sup_{\lambda>0}\iint_{Q_\tau\cap[u>\lambda]}u^{m-1}|D(u-k)_-|^p\zeta^pdz<+\infty.$$
Finally, we use monotone convergence theorem again to the integral involving gradient, there holds
\begin{equation*}\begin{split}\iint_{Q_\tau\cap[u>\lambda]}u^{m-1}|D(u-k)_-|^p\zeta^pdz\to&\iint_{Q_\tau\cap [u>0]}u^{m-1}|D(u-k)_-|^p\zeta^pdz\quad\mathrm{as}\quad\lambda\downarrow 0.\end{split}\end{equation*}
Taking into account that $u>0$ almost everywhere in $Q_T$, we arrive at
\begin{equation*}\begin{split}\iint_{Q_\tau}u^{m-1}&|D(u-k)_-|^p\zeta^pdz+\frac{1}{2}\int_{Q}(u-k)_-^2(\cdot,\tau)\zeta^p(\cdot,\tau)dx\\ \leq &\frac{1}{2}\iint_{Q_\tau}(u-k)_-^2\left|\frac{\partial \zeta^p}{\partial t}\right|dz\\&+\frac{1}{2}\int_{Q}(u-k)_-^2(\cdot,t_1)\zeta^p(\cdot,\tau)dx\\&\quad+\gamma \iint_{Q_\tau}u^{m-1}(u-k)_-^p|D\zeta|^pdz,\end{split}\end{equation*}
and taking the supremum over $\tau\in (t_1,t_2)$ proves the Proposition.
\end{proof}
\begin{proposition}\label{C2}Let $m+p>3$ and $0<m<1$. Let $u$ be a positive, weak supersolution to \eqref{DDP}-\eqref{mp} in an open set which compactly contains $Q_T$. There exists a positive constant $\gamma$ depending only upon $m$, $p$, $C_0$ and $C_1$,
such that for every cylinder $Q_{t_1,t_2}=Q\times(t_1,t_2)\subset Q_T$, every $k>0$ and every piecewise smooth, nonnegative cutoff function $\zeta$ vanishing on $\partial Q\times[t_1,t_2]$,
\begin{equation}\begin{split}\label{CCC}\esssup_{t_1\leq \tau\leq t_2}\int_{Q}& (u-k)_-^2(\cdot,\tau)\zeta^p(\cdot,\tau)dx+k^{m-1}\iint_{Q_{t_1,t_2}}|D(u-k)_-|^p\zeta^pdz
\\ \leq&\gamma k\int_Q(u-k)_-(\cdot,t_1)\zeta^p(\cdot,t_1)dx\\&+\gamma k\iint_{Q_{t_1,t_2}}(u-k)_-\left|\frac{\partial \zeta}{\partial t}\right|dz\\&\quad+\gamma k^{m+p-1}\iint_{Q_{t_1,t_2}}\chi_{[u\leq k]}|D\zeta|^pdz.\end{split}\end{equation}
\end{proposition}
\begin{proof} To start with, we fix a $\tau\in(t_1,t_2)$ and define $Q_\tau=Q\times(t_1,\tau)$.
Let $\theta_\nu(t)$ be as in \eqref{theta time} with $t_2=\tau$.
We observe from the assumption $m+p>3$ and $0<m<1$ that $\alpha<1$. This motivates us to choose the testing function \begin{equation}\label{CCC1}\Phi(x,t)=\left[u(x,t)^\alpha-k^\alpha\right]_-\zeta(x,t)^p\theta_\nu(t)\end{equation}
in the weak form \eqref{averaged equation}. From \eqref{regularity}, we check that $\phi\in L^p(0,T;W_0^{1,p}(Q))$. Therefore, $\phi$ is an admissible test function and this yields
\begin{equation}\begin{split}\label{CCC2}\iint_{Q_\tau}&\left[A(x,t,u,Du)\right]^*\cdot D\left[(u^\alpha-k^\alpha)_-\zeta^p\theta_\nu\right]dz
+\iint_{Q_\tau}(u^\alpha-k^\alpha)_-\zeta^p\theta_\nu\frac{\partial u^*}{\partial t} dz
\geq0.\end{split}\end{equation}
Concerning the second integral on the left hand side, we use \eqref{averaged function 1} to write
\begin{equation}\begin{split}\label{CCC3}\iint_{Q_\tau}&(u^\alpha-k^\alpha)_-\zeta^p\theta_\nu\frac{\partial u^*}{\partial t} dz
\\=&\iint_{Q_\tau}[(u^*)^\alpha-k^\alpha]_-\zeta^p\theta_\nu\frac{\partial u^*}{\partial t} dz
\\&+\iint_{Q_\tau}\zeta^p\theta_\nu\left[(u^\alpha-k^\alpha)_--((u^*)^\alpha-k^\alpha)_-\right]\frac{u-u^*}{\sigma} dz.\end{split}\end{equation}
Notice that the function $f(s)=(s^\alpha-k^\alpha)_-$ is decreasing with respect to $s$, we have
\begin{equation}\label{CCC4}\iint_{Q_\tau}\zeta^p\theta_\nu\left[(u^\alpha-k^\alpha)_--((u^*)^\alpha-k^\alpha)_-\right]\frac{u-u^*}{\sigma} dz\leq0.\end{equation}
Combining \eqref{CCC2}-\eqref{CCC4}, we arrive at
\begin{equation}\begin{split}\label{CCC5}\iint_{Q_\tau}[(u^*)^\alpha-k^\alpha]_-\zeta^p\theta_\nu\frac{\partial u^*}{\partial t} dz
+\iint_{Q_\tau}&\left[A(x,t,u,Du)\right]^*\cdot D\left[(u^\alpha-k^\alpha)_-\zeta^p\theta_\nu\right]dz\geq0.\end{split}\end{equation}
To proceed further, we use the method as employed in the proof of \cite[Proposition 2.1]{Gianazza} to deal with the first integral on the left hand side of \eqref{CCC5}.
We first observe that
$$\frac{\partial u^*}{\partial t}[(u^*)^\alpha-k^\alpha]_-=-\frac{\partial}{\partial t}\int_{u^*}^{k}(k^\alpha-s^\alpha)_+ds.$$
Using integration by parts, the above identity yields
\begin{equation*}\begin{split}\label{k4}\iint_{Q_\tau}&[(u^*)^\alpha-k^\alpha]_-\zeta^p\theta_\nu\frac{\partial u^*}{\partial t} dz
\\=&\iint_{Q_\tau}p\zeta^{p-1}\left(\int_{u^*}^{k}(k^\alpha-s^\alpha)_+ds\right)\frac{\partial \zeta}{\partial t}\theta_\nu dz
+\iint_{Q_\tau}\zeta^{p}\left(\int_{u^*}^{k}(k^\alpha-s^\alpha)_+ds\right)\frac{\partial \theta_\nu}{\partial t} dz.\end{split}\end{equation*}
Since $0<\alpha<1$, it follows  that
\begin{equation*}\int_{u^*}^{k}(k^\alpha-s^\alpha)_+ds\geq \frac{1}{2}\alpha k^{\alpha-1}(u^*-k)_-^2\quad\mathrm{and}\quad\int_{u^*}^{k}(k^\alpha-s^\alpha)_+ds\leq k^\alpha(u^*-k)_-,\end{equation*}
where the proof can be found at \cite[page 36]{Gianazza}.
Then we conclude from \eqref{theta time} that
\begin{equation}\begin{split}\label{CCC6}\iint_{Q_\tau}&[(u^*)^\alpha-k^\alpha]_-\zeta^p\theta_\nu\frac{\partial u^*}{\partial t} dz
\\& \leq pk^\alpha\iint_{Q_\tau}(u^*-k)_-\zeta^{p-1}\theta_\nu\left|\frac{\partial \zeta}{\partial t}\right|dz
\\ &\quad +\frac{k^{\alpha}}{\nu}\int_{t_1}^{t_1+\nu}\int_{Q}(u^*-k)_-\zeta^pdz
-\frac{\alpha k^{\alpha-1}}{2\nu}\int_{\tau-\nu}^{\tau}\int_{Q} (u^*-k)_-^2\zeta^pdz.\end{split}\end{equation}
At this stage, we combine the estimates \eqref{CCC5}-\eqref{CCC6} and pass to the limits $\sigma\downarrow 0$ and $\nu\downarrow0$. We arrive at
\begin{equation}\begin{split}\label{CCC7}\iint_{Q_\tau}&A(x,t,u,Du)\cdot D\left((u^\alpha-k^\alpha)_-\zeta^p\right)dz
\\ &+ pk^\alpha\iint_{Q_\tau}(u-k)_-\zeta^{p-1}\left|\frac{\partial \zeta}{\partial t}\right|dz
+k^{\alpha}\int_{Q}(u-k)_-(\cdot,t_1)\zeta^p(\cdot,t_1)dx
\\&\geq\frac{1}{2}\alpha k^{\alpha-1}\int_{Q} (u-k)_-^2(\cdot,\tau)\zeta^p(\cdot,\tau)dx.\end{split}\end{equation}
To estimate the first term on left hand side of \eqref{CCC7}, we use \eqref{structure def} and Young's inequality to obtain
\begin{equation}\begin{split}\label{CCC8}\iint_{Q_\tau}&A(x,t,u,Du)\cdot D\left((u^\alpha-k^\alpha)_-\zeta^p\right)dz
\\ \leq&-\iint_{Q_\tau}u^{\alpha+m-2}|Du|^p\zeta^p\chi_{[u< k]}dz+\gamma k^{\alpha p}\iint_{Q_\tau}|D\zeta|^p\chi_{[u\leq k]}dz\\ \leq&-k^{\alpha+m-2}\iint_{Q_\tau}|D(u-k)_-|^p\zeta^pdz+\gamma k^{\alpha p}\iint_{Q_\tau}|D\zeta|^p\chi_{[u\leq k]}dz,\end{split}\end{equation}
since $m+\alpha-2=p\frac{m-1}{p-1}<0$.
Inserting the estimate \eqref{CCC8} into \eqref{CCC7} and dividing by $\frac{1}{2}\alpha k^{\alpha-1}$ gives
\begin{equation*}\begin{split}k^{m-1}\iint_{Q_\tau}&|D(u-k)_-|^p\zeta^pdz+\int_Q(u-k)_-^2(\cdot,\tau)\zeta^p(\cdot,\tau)dx\\ \leq&\gamma k^{\alpha p-\alpha+1}\iint_{Q_\tau}|D\zeta|^p\chi_{[u<k]}dz\\&+\gamma k\iint_{Q_\tau}(u-k)_-\zeta^{p-1}\left|\frac{\partial \zeta}{\partial t}\right|dz\\&\quad+\gamma k\int_Q(u-k)_-(\cdot,t_1)\zeta^p(\cdot,t_1)dx.\end{split}\end{equation*}
Since  $\alpha p-\alpha+1=m+p-1$, the previous estimate yields \eqref{CCC}.
\end{proof}
Up to now we have studied Caccioppoli estimates involving the truncated functions $(u-k)_-$, which will be used in Section \ref{Expansion of Positivity}-\ref{Hot alternative}. However, in order to perform the Moser's iteration, we need Caccioppoli estimates of the following type. Our proof is in the spirit of \cite[Lemma 2.4]{Le}.
Since the weak supersolution multiplied by a cut-off function cannot be used as a testing function in the weak form \eqref{averaged equation}, we have to perform a refined analysis
to overcome this difficulty.
\begin{proposition}\label{C3}Let $p>2$, $m+p>3$ and $\epsilon\in(-1,0)$. Suppose that $u$ is a positive, weak supersolution to \eqref{DDP}-\eqref{mp} in an open set which compactly contains $Q_T$. There exists a positive constant $\gamma$ depending only upon $m$, $p$, $C_0$ and $C_1$,
such that for every cylinder $Q_{t_1,t_2}=Q\times(t_1,t_2)\subset Q_T$ and every non-negative function $\varphi\in C_0^\infty(Q_T)$,
\begin{equation}\begin{split}\label{moser1}\frac{1}{\epsilon(1+\epsilon)}&\int_{Q}u(\cdot,t_2)^{1+\epsilon}\varphi(\cdot,t_2)^pdx
-\frac{1}{\epsilon(1+\epsilon)}\int_{Q}u(\cdot,t_1)^{1+\epsilon}\varphi(\cdot,t_1)^pdx\\&+\iint_{Q_{t_1,t_2}}|Du|^pu^{m+\epsilon-2}\varphi^pdz
\\ \leq&\gamma \left(\frac{1}{|\epsilon|^p}\iint_{Q_{t_1,t_2}}u^{m+p+\epsilon-2}|D\varphi|^pdz+ \iint_{Q_{t_1,t_2}}u^{1+\epsilon}\frac{1}{\epsilon(1+\epsilon)}\frac{\partial \varphi^p}{\partial t}dz\right).\end{split}\end{equation}
If, in addition, $\varphi=0$ on $Q\times\{t_2\}$, then
\begin{equation}\begin{split}\label{moser2}\frac{1}{|\epsilon(1+\epsilon)|}&\esssup_{t_1\leq \tau\leq t_2}\int_{Q}u^{1+\epsilon}\varphi^pdx+\iint_{Q_{t_1,t_2}}|Du|^pu^{m+\epsilon-2}\varphi^pdz
\\&\leq\gamma \left(\frac{1}{|\epsilon|^p}\iint_{Q_{t_1,t_2}}u^{m+p+\epsilon-2}|D\varphi|^pdz+ \iint_{Q_{t_1,t_2}}u^{1+\epsilon}\left(\frac{1}{\epsilon(1+\epsilon)}\frac{\partial \varphi^p}{\partial t}\right)_+dz\right).\end{split}\end{equation}
\end{proposition}
\begin{proof}
Our first destination is to construct a suitable testing function. Let $\delta=-\frac{\epsilon}{\alpha}$. For $\lambda>0$, we invoke the auxiliary function
\begin{equation}\label{CCCC1}
	H_\lambda(s)=\begin{cases}
	\lambda^{-\delta}+\delta\lambda^{-1-\delta}(\lambda-s),&0\leq s\leq\lambda,\\
	s^{-\delta},&s>\lambda,
	\end{cases}
\end{equation}
which was first introduced by Ivert, Marola and Masson \cite{IMM}.
Furthermore, it is easy to verify that
\begin{equation}\label{CCCC2}
	H_\lambda^\prime(s)=\begin{cases}
	-\delta\lambda^{-1-\delta},&0\leq s\leq\lambda,\\
	-\delta s^{-1-\delta},&s>\lambda.
	\end{cases}
\end{equation}
Then we see that $H_\lambda(s)$ is a Lipschitz function. By \cite[Theorem 2.1.11]{Z}, we conclude that
$$D\left[H_\lambda(u^\alpha)\right]=H_\lambda^\prime(u^\alpha)D(u^\alpha)$$
and therefore $H_\lambda(u^\alpha)\in L^p(0,T;W^{1,p}(Q))$.
Let $\tau_1$, $\tau_2$ be the fixed time levels in the interval $(t_1,t_2)$ with $\tau_1<\tau_2$. Let $\theta_\nu(t)$ be as in \eqref{theta time} with $t_1$, $t_2$ replaced by $\tau_1$, $\tau_2$.
We are now in a position to introduce a testing function
\begin{equation*}\Phi(x,t)=H_\lambda(u^\alpha)\varphi(x,t)^p\theta_\nu(t)\end{equation*}
in the weak form \eqref{averaged equation}. Observe that $\phi\in L^p(0,T;W_0^{1,p}(Q))$ and there holds
 \begin{equation}\begin{split}\label{CCCC3}\iint_{Q_T}&H_\lambda(u^\alpha)\varphi^p\theta_\nu\frac{\partial u^*}{\partial t} dz+\iint_{Q_T}pH_\lambda(u^\alpha)\varphi^{p-1}\theta_\nu\left[A(x,t,u,Du)\right]^*\cdot D\varphi dz
 \\&\geq-\iint_{Q_T}H_\lambda^\prime(u^\alpha)\varphi^p\theta_\nu\left[A(x,t,u,Du)\right]^*\cdot Du^\alpha dz
.\end{split}\end{equation}
We now apply \eqref{averaged function 1} to decompose the first term on the left hand side of \eqref{CCCC3} as follows:
\begin{equation}\begin{split}\label{CCC4}&\iint_{Q_T}H_\lambda(u^\alpha)\varphi^p\theta_\nu\frac{\partial u^*}{\partial t} dz\\&=\iint_{Q_T}H_\lambda((u^*)^\alpha)\varphi^p\theta_\nu\frac{\partial u^*}{\partial t} dz+\iint_{Q_T}\left[H_\lambda(u^\alpha)-H_\lambda((u^*)^\alpha)\right]\varphi^p\theta_\nu\frac{u- u^*}{\sigma} dz
\end{split}\end{equation}
Since $H_\lambda(s^\alpha)$ is decreasing with respect to $s$, then
 \begin{equation*}\iint_{Q_T}\left[H_\lambda(u^\alpha)-H_\lambda((u^*)^\alpha)\right]\varphi^p\theta_\nu\frac{u- u^*}{\sigma} dz\leq 0.\end{equation*}
The estimate above together with \eqref{CCCC3} yield
\begin{equation}\begin{split}\label{CCCC5}\iint_{Q_T}H_\lambda(u^\alpha)\varphi^p\theta_\nu\frac{\partial u^*}{\partial t} dz\leq \iint_{Q_T}H_\lambda((u^*)^\alpha)\varphi^p\theta_\nu\frac{\partial u^*}{\partial t} dz.\end{split}\end{equation}
Set $h_\lambda(s)=\int_0^sH_\lambda(r^\alpha)dr$. Then we infer that
$$\frac{\partial h_\lambda(u^*)}{\partial t}=H_\lambda((u^*)^\alpha)\frac{\partial u^*}{\partial t}.$$
Using the above identity and integration by parts, we obtain
\begin{equation}\begin{split}\label{CCCC6}\iint_{Q_T}H_\lambda((u^*)^\alpha)\varphi^p\theta_\nu\frac{\partial u^*}{\partial t} dz
=-\iint_{Q_T}h_\lambda(u^*)\varphi^p\frac{\partial\theta_\nu}{\partial t} dz-\iint_{Q_T}h_\lambda(u^*)\theta_\nu\frac{\partial\varphi^p}{\partial t} dz.\end{split}\end{equation}
Notice that $h_\lambda(\cdot)$ is a Lipschitz function and this enable us to pass to the limit $\sigma\downarrow 0$ in \eqref{CCCC6}.
Combining \eqref{CCCC3}-\eqref{CCCC6}, we pass to the limits $\sigma\downarrow 0$ and $\nu\downarrow 0$. This implies that
 \begin{equation}\begin{split}\label{CCCC7}\int_{Q}h_\lambda(u)(\cdot,\tau)& \varphi(\cdot,\tau) ^pdx-\int_{Q}h_\lambda(u)(\cdot,\tau_1) \varphi(\cdot,\tau_1) ^pdx-\iint_{Q_{\tau_1,\tau_2}}h_\lambda(u)\frac{\partial\varphi^p}{\partial t} dz\\&\quad+\iint_{Q_{\tau_1,\tau_2}}pH_\lambda(u^\alpha)\varphi^{p-1} A(x,t,u,Du)\cdot D\varphi dz
 \\&\geq-\iint_{Q_{\tau_1,\tau_2}}H_\lambda^\prime(u^\alpha)\varphi^p A(x,t,u,Du)\cdot Du^\alpha dz,\end{split}\end{equation}
where $Q_{\tau_1,\tau_2}=Q\times(\tau_1,\tau_2)$.
At this stage,
we use \eqref{structure def} and recalling that $H_\lambda^\prime(u)\leq0$, there holds
\begin{equation}\begin{split}\label{CCCC8}\iint_{Q_{\tau_1,\tau_2}}&H_\lambda^\prime(u^\alpha)\varphi^p A(x,t,u,Du)\cdot Du^\alpha dz
\leq \alpha C_0\iint_{Q_{\tau_1,\tau_2}}H_\lambda^\prime(u^\alpha)\varphi^p u^{m+\alpha-2}|Du|^p dz.\end{split}\end{equation}
In order to estimate the fourth term on the left hand side, we use \eqref{structure def} and Young's inequality to obtain
\begin{equation}\begin{split}\label{CCCC9}&\left|pH_\lambda(u^\alpha)\varphi^{p-1} A(x,t,u,Du)\cdot D\varphi \right|
\\&\leq \frac{(C_1p)^p}{pC_0^{p-1}}\frac{H_\lambda(u^\alpha)^p}{|H_\lambda^\prime(u^\alpha)|^{p-1}}|D\varphi|^p
+\frac{p-1}{p}C_0u^{(m-1)\frac{p}{p-1}}|H_\lambda^\prime(u^\alpha)|\varphi^p|Du|^p.\end{split}\end{equation}
Inserting \eqref{CCCC8} and \eqref{CCCC9} into \eqref{CCCC7}, and noting that $m+\alpha-2=p\frac{m-1}{p-1}$, we arrive at
 \begin{equation}\begin{split}\label{CCCC10}-\int_{Q}&h_\lambda(u)(\cdot,\tau_2) \varphi(\cdot,\tau_2) ^pdx+\int_{Q}h_\lambda(u)(\cdot,\tau_1) \varphi(\cdot,\tau_1) ^pdx\\&+\iint_{Q_{\tau_1,\tau_2}}|H_\lambda^\prime(u^\alpha)|\varphi^p u^{m+\alpha-2}|Du|^p dz
 \\&\leq \gamma\iint_{Q_{\tau_1,\tau_2}}\left(\frac{H_\lambda(u^\alpha)^p}{|H_\lambda^\prime(u^\alpha)|^{p-1}}|D\varphi|^p+h_\lambda(u)\left(-\frac{\partial\varphi^p}{\partial t}\right)_+\right)dz,\end{split}\end{equation}
 where $\gamma$ depends only upon $m$, $p$, $C_0$ and $C_1$. The task now is to show the convergence of the integrals in \eqref{CCCC10} as $\lambda\downarrow 0$.
 We first observe that $H_\lambda(s)$, $|H_\lambda^\prime(s)|$ and $h_\lambda(s)$ are decreasing with respect to $\lambda$ (cf. \cite[Lemma 2.4]{Le}).
By monotone convergence theorem,
 \begin{equation}\label{CCCC11}\lim_{\lambda\to0}h_\lambda(u)=\lim_{\lambda\to0}\int_0^uH_\lambda(r^\alpha)dr=\int_0^u\lim_{\lambda\to0}H_\lambda(r^\alpha)dr=\frac{u^{1+\epsilon}}
 {1+\epsilon}.\end{equation}
Furthermore, we obtain from \eqref{CCCC11} that
\begin{equation}\begin{split}\label{CCCC12}&\int_{Q}h_\lambda(u)(\cdot,\tau_i) \varphi(\cdot,\tau_i) ^pdx\to\int_{Q}\frac{u(\cdot,\tau_i)^{1+\epsilon}}
 {1+\epsilon}\varphi(\cdot,\tau_i) ^pdx,\quad i=1,2,\quad\mathrm{and} \\ &\iint_{Q_{\tau_1,\tau_2}}h_\lambda(u)\left(-\frac{\partial\varphi^p}{\partial t}\right)_+dz\to\iint_{Q_{\tau_1,\tau_2}}\frac{u^{1+\epsilon}}
 {1+\epsilon}\left(-\frac{\partial\varphi^p}{\partial t}\right)_+dz
\end{split}\end{equation}
as $\lambda\downarrow0$, by the monotone convergence theorem.
To handle the first term on the right hand side of \eqref{CCCC11}, we calculate
\begin{equation*}
	T(\lambda):=\frac{H_\lambda(u^\alpha)^p}{|H_\lambda^\prime(u^\alpha)|^{p-1}}=\begin{cases}
	\delta^{1-p}\lambda^{-1-\delta}[\lambda+\delta(\lambda-u^\alpha)]^p,&\lambda\geq u^\alpha,\\
	\delta^{1-p}u^{m+p-2+\epsilon},&0<\lambda<u^\alpha.
	\end{cases}
\end{equation*}
Since $m+p>3$ and $-1<\epsilon<0$, then $0<\delta< p-1$. This implies that $T(\lambda)$ is increasing with respect to $\lambda$, which yields
$$\frac{H_\lambda(u^\alpha)^p}{|H_\lambda^\prime(u^\alpha)|^{p-1}}\leq \frac{H_1(u^\alpha)^p}{|H_1^\prime(u^\alpha)|^{p-1}}$$
for $0<\lambda<1$.
From \eqref{regularity}, we check that $H_1(u^\alpha)^p/|H_1^\prime(u^\alpha)|^{p-1}$ is an integrable function and therefore
\begin{equation}\begin{split}\label{CCCC13}\iint_{Q_{\tau_1,\tau_2}}\frac{H_\lambda(u^\alpha)^p}{|H_\lambda^\prime(u^\alpha)|^{p-1}}|D\varphi|^pdz\to
\delta^{1-p}\iint_{Q_{\tau_1,\tau_2}}u^{m+p-2+\epsilon}|D\varphi|^pdz,\end{split}\end{equation}
as $\lambda\downarrow0$, by Lebesgue dominated convergence theorem. At this stage, we conclude from \eqref{CCCC10}, \eqref{CCCC12} and \eqref{CCCC13} that
$$\sup_{\lambda>0}\iint_{Q_{\tau_1,\tau_2}}|H_\lambda^\prime(u^\alpha)|\varphi^p u^{m+\alpha-2}|Du|^p dz<+\infty.$$
We use the monotone convergence theorem again,
\begin{equation}\begin{split}\label{CCCC14}\iint_{Q_{\tau_1,\tau_2}}|H_\lambda^\prime(u^\alpha)|\varphi^p u^{m+\alpha-2}|Du|^p dz
 \to \delta\iint_{Q_{\tau_1,\tau_2}}\varphi^p u^{m+\epsilon-2}|Du|^p dz\end{split}\end{equation}
as $\lambda\downarrow 0$. Taking into account \eqref{CCCC11}-\eqref{CCCC14}, we pass to the limit $\lambda\downarrow 0$ in \eqref{CCCC10} and this implies that
 \begin{equation}\begin{split}\label{CCCC15}-\int_{Q}&\frac{1}{1+\epsilon}u^{1+\epsilon}(\cdot,\tau_2) \varphi(\cdot,\tau_2) ^pdx+\int_{Q}\frac{1}{1+\epsilon}u^{1+\epsilon}(\cdot,\tau_1) \varphi(\cdot,\tau_1) ^pdx\\&\quad +\delta\iint_{Q_{\tau_1,\tau_2}}\varphi^p u^{m+\epsilon-2}|Du|^p dz
 \\&\leq \gamma\iint_{Q_{\tau_1,\tau_2}}\left(\delta^{1-p}u^{m+p-2+\epsilon}|D\varphi|^p+\frac{1}{1+\epsilon}u^{1+\epsilon}\left(-\frac{\partial\varphi^p}{\partial t}\right)_+\right)dz,\end{split}\end{equation}
 which proves \eqref{moser1} by setting $\tau_1=t_1$ and $\tau_2=t_2$. Moreover, if $\varphi=0$ on $Q\times\{t_2\}$, then we choose $\tau_2=t_2$ and take the supremum over $\tau_1\in(t_1,t_2)$ in \eqref{CCCC15}. This proves estimate \eqref{moser2}. The proof of the proposition is now complete.
\end{proof}
\section{Expansion of Positivity}\label{Expansion of Positivity}
In this section, we will prove the expansion positivity property for positive, weak supersolutions of doubly degenerate parabolic equations. This property asserts that
some positive information of $u$ at some time level $s$, over the cube $K_\rho(y)$, propagates to further times $s+\theta\rho^2$, and a larger cube $K_{2\rho}(y)$. Here and
subsequently, $Q_\rho^{\pm}(\theta)$ stand for "forward" and "backward" parabolic cylinders of the form
\begin{equation*}Q_\rho^{+}(\theta)=K_\rho\times(-\theta\rho^p,0],\quad Q_\rho^{-}(\theta)=K_\rho\times(0,\theta\rho^p],\end{equation*}
where $\theta$ is a positive parameter.
To start with, we prove the following De-Giorgi type Lemma.
\begin{lemma}\label{Q>} Let $u$ be a positive, weak supersolution of \eqref{DDP}-\eqref{mp} in an open set which compactly contains $Q_{2\rho}^-(\theta)$ and
$a\in(0,1)$ be a fixed number. There exists a number $\nu_-$ depending only upon $m$, $p$, $C_0$, $C_1$ and $a$, such that if
\begin{equation}\label{Q>a}|[u\leq M]\cap Q_{2\rho}^-(\theta)|\leq\nu_-|Q_{2\rho}^-(\theta)|\end{equation}
then
\begin{equation}\label{Q>c}u\geq aM\quad\mathrm{a.\ e.}\quad\mathrm{in}\quad Q_{\rho}^-(\theta).\end{equation}
\end{lemma}
\begin{proof} The argument is standard and we just sketch the proof. For $n=0,1,2,\cdots$ set
$$\rho_n=\rho+2^{-n}\rho\quad\quad\mathrm{and}\quad\quad k_n=aM+(1-a)2^{-n}M.$$ We define $K_n=K_{\rho_n}$ and $Q_n=K_n\times(-\theta \rho_n^p,0]$. The cut-off function
$\zeta$ is chosen of the form $\zeta(x,t)=\phi(x)\theta(t)$ where
\begin{equation*}
	\phi(x)=\begin{cases}
	1,&\mathrm{in}\quad K_{n+1},\\
	0,&\mathrm{in}\quad \mathbb{R}^N\backslash K_{n},
	\end{cases}
\quad|D\phi(x)|\leq\frac{2^{n+1}}{\rho},\end{equation*}
\begin{equation*}
	\quad\quad\ \theta(t)=\begin{cases}
	0,&\mathrm{for}\quad  t<-\theta\rho_n^p,\\
	1,&\mathrm{for}\quad t\geq -\theta\rho_{n+1}^p,
	\end{cases}
\quad|\theta^\prime(t)|\leq\frac{2^{p(n+1)}}{\theta\rho^p}.\end{equation*}
To proceed further, we first consider the case $m+p>3$ and $0<m<1$.
We apply the Caccioppoli estimate \eqref{CCC} on the cylinder $Q_n$ for $(u-k_n)_-$ to obtain
\begin{equation}\begin{split}\label{DG1}\esssup_{-\theta\rho_n^p< t\leq 0}&\int_{K_n} (u-k_n)_-^2\zeta^pdx+k_n^{m-1}\int_{-\theta\rho_n^p}^0\int_{K_n}|D(u-k_n)_-|^p\zeta^pdz
\\ &\leq\gamma k_n\frac{2^{p(n+1)}}{\theta\rho^p}\iint_{Q_n}(u-k_n)_-dz+\gamma k_n^{m+p-1}\iint_{Q_n}\chi_{[u\leq k_n]}|D\phi|^pdz
\\ &\leq\gamma\frac{2^{np}M^{m+p-1}}{\rho^p}\left(1+\frac{1}{\theta M^{m+p-3}}\right)\left|Q_n\cap [u<k_n]\right|.\end{split}\end{equation}
At this stage, we use Proposition \ref{parabolic sobolev} together with \eqref{DG1} to infer that
\begin{equation*}\begin{split}\iint_{Q_n}&\left[(u-k_n)_-\zeta\right]^{p\frac{N+2}{N}}dxdt\leq
\gamma\left(1+\frac{1}{\theta M^{m+p-3}}\right)^{1+\frac{p}{N}}\frac{2^{np(1+\frac{p}{N})}M^{(m+p-1)\frac{p}{N}+p}}{\rho^{\frac{1}{N}p^2+p}}
\left|Q_n\cap [u<k_n]\right|^{1+\frac{p}{N}}.\end{split}\end{equation*}
Moreover, we estimate the left hand side from below
\begin{equation*}\begin{split}&\iint_{Q_n}\left[(u-k_n)_-\zeta\right]^{p\frac{N+2}{N}}dxdt\geq \left(\frac{1-a}{2^{n+1}}M\right)^{p\frac{N+2}{N}}\left|Q_{n+1}\cap [u<k_{n+1}]\right|.\end{split}\end{equation*}
Set
$$Y_n=\frac{\left|Q_n\cap [u<k_n]\right|}{|Q_n|}.$$
The previous estimates now yield
\begin{equation*}\begin{split}Y_{n+1}\leq \gamma\left(1+\frac{1}{\theta M^{m+p-3}}\right)^{1+\frac{p}{N}}\frac{\left(\theta M^{m+p-3}\right)^{\frac{p}{N}}}{(1-a)^{p\frac{N+2}{N}}}b^nY_n^{1+\frac{p}{N}},\end{split}\end{equation*}
where $b=2^{\frac{p^2}{N}+p+p\frac{N+2}{N}}$ and $\gamma$ depends only upon $N$, $m$, $p$, $C_0$ and $C_1$. From Lemma 4.1 of Chapter I in \cite{Di}, it follows that $Y_n\to0$ as $n\to\infty$, provided $Y_0\leq \nu_-$ where
\begin{equation}\label{nu-m<1}\nu_-=\gamma^{-\frac{N}{p}} b^{-\frac{N^2}{p^2}}(1-a)^{N+2}\frac{(\theta M^{m+p-3})^{\frac{N}{p}}}{(1+\theta M^{m+p-3})^{1+\frac{N}{p}}}.\end{equation}
Finally, using the same argument as in the proof of \cite[Lemma 3.3]{FS}, we can easily carry out the proof of this lemma in the case $p>2$ and $m>1$, with $\nu_-$ defined in \eqref{nu-m<1}. We omit the details.
\end{proof}
The next lemma deals with the expansion of positivity which translates a positive information to further times but within the same cube.
\begin{lemma}\label{K>} Let $u$ be a positive, weak supersolution of \eqref{DDP}-\eqref{mp} in an open set which compactly contains $E_T$. Suppose that for $(y,s)\in E_T$ and $\rho>0$ there holds
\begin{equation}\label{Ka}\left|[u(\cdot,s)\geq M]\cap K_\rho(y)\right|\geq \mu|K_\rho(y)|\end{equation}
for some $M>0$ and some $\mu\in(0,1)$. There exist $\kappa$ and $\delta$ in $(0,1)$, depending only upon the data $p$, $m$, $N$, $C_0$, $C_1$ and $\mu$, independent of $M$, such that
\begin{equation}\label{KK}\left|[u(\cdot,t)\geq \kappa M]\cap K_\rho(y)\right|\geq \frac{1}{2}\mu|K_\rho(y)|\quad\mathrm{for\ all}\quad t\in\left(s,s+\delta M^{3-m-p}\rho^p\right].\end{equation}
\end{lemma}
\begin{proof}
Without loss of generality, we may assume $(y,s)=(0,0)$. This is always possible by using a translation.  For $k>0$ and $t>0$ set $A_{k,\rho}(t)=[u(\cdot,t)<k]\cap K_\rho.$
The assumption \eqref{Ka} yields $|A_{M,\rho}(0)|\leq (1-\mu)|K_\rho|$.

We now apply Caccioppoli estimates for the truncated functions $(u-M)_-$
over cylinder $K_\rho\times(0,\theta \rho^p]$ where $\theta>0$ is to be determined. The cut-off function $\zeta$ is taken of the form $\zeta=\zeta(x)$ and such that $\zeta=1$
on $K_{(1-\sigma)\rho}$, and $|D\zeta|\leq (\sigma\rho)^{-1}$.
In the case $m+p>3$, $0<m<1$, we use Caccioppoli estimate \eqref{CCC} to obtain
\begin{equation}\begin{split}\label{K1}\esssup_{t\in(0,\theta\rho^p]}&\int_{K_{(1-\sigma)\rho}}(u-M)_-^2(\cdot,t)dx\\&\leq \gamma M\int_{K_{\rho}}(u-M)_-(\cdot,0)\zeta ^p(\cdot,0)dx+\gamma M^{m+p-1}\int_{0}^{\theta \rho^p}\int_{K_\rho}\chi_{[u<M]}|D\zeta|^pdxdt
\\&\leq \gamma M\int_{K_{\rho}}(u-M)_-(\cdot,0)dx+\gamma\frac{M^{m+p-1}}{(\sigma \rho)^p}\int_{0}^{\theta \rho^p}\int_{K_\rho}\chi_{[u<M]}dxdt\\ &\leq
\gamma M^2\left((1-\mu)+ \frac{\theta M^{m+p-3}}{\sigma ^p}\right)|K_\rho|.\end{split}\end{equation}
In the case $p>2$, $m>1$, we use Caccioppoli estimate \eqref{CC} to obtain
\begin{equation}\begin{split}\label{K2}\esssup_{t\in(0,\theta\rho^p]}&\int_{K_{(1-\sigma)\rho}}(u-M)_-^2(\cdot,t)dx\\&\leq \tfrac{1}{2}\int_{K_{\rho}}(u-M)_-^2(\cdot,0)\zeta ^p(\cdot,0)dx+\gamma \int_{0}^{\theta \rho^p}\int_{K_\rho}u^{m-1}(u-M)_-^p|D\zeta|^pdxdt
\\&\leq \tfrac{1}{2}\int_{K_{\rho}}(u-M)_-^2(\cdot,0)dx+\gamma\frac{M^{m+p-1}}{(\sigma \rho)^p}\int_{0}^{\theta \rho^p}\int_{K_\rho}\chi_{[u<M]}dxdt\\ &\leq
\gamma M^2\left((1-\mu)+ \frac{\theta M^{m+p-3}}{\sigma ^p}\right)|K_\rho|.\end{split}\end{equation}
Combining estimates \eqref{K1} and \eqref{K2}, we conclude that, in either case, the estimate
\begin{equation*}\begin{split}\int_{K_{(1-\sigma)\rho}}&(u-M)_-^2(\cdot,t)dx\leq \gamma M^2\left((1-\mu)+ \frac{\theta M^{m+p-3}}{\sigma ^p}\right)|K_\rho|\end{split}\end{equation*}
holds for all $t\in(0,\theta\rho^p]$. Furthermore, we estimate the left hand side from below
\begin{equation*}\begin{split}\int_{K_{(1-\sigma)\rho}}(u-M)_-^2(\cdot,t)dx&\geq \int_{K_{(1-\sigma)\rho}\cap [u<\kappa M]}(u-M)_-^2(\cdot,t)dx\geq M^2(1-\kappa)^2|A_{\kappa M,(1-\sigma)\rho}(t)|,\end{split}\end{equation*}
where $\kappa$ is to be determined later.
Recalling from the proof of Lemma 1.1 in \cite[chapter 2]{DGV} that $|A_{\kappa M,\rho}(t)|\leq |A_{\kappa M,(1-\sigma)\rho}(t)|+N\sigma|K_\rho|$, we obtain
\begin{equation*}\begin{split}|A_{\kappa M,\rho}(t)|\leq&\frac{1}{M^2(1-\kappa)^2}\int_{K_{(1-\sigma)\rho}}(u-M)_-^2(\cdot,t)dx+N\sigma|K_\rho|
\\ \leq&\frac{1}{(1-\kappa)^2}\left(1-\mu+\frac{\gamma}{\sigma^p}\theta M^{m+p-3}+N\sigma\right)|K_\rho|.\end{split}\end{equation*}
We have established an estimate for the measure of level set, and now we choose
\begin{equation}\begin{split}\label{parameter1}\theta=\delta M^{3-p-m},\quad\sigma=\frac{\mu}{64N},\quad \delta=\frac{\mu^{p+1}}{4^{3p+4}\gamma N}\quad\mathrm{and} \quad\kappa=\frac{1}{64}\mu.\end{split}\end{equation}
Then for such choices the lemma follows.
\end{proof}
With the help of the preceding two lemmas, we can now prove the main result of this section. Our proof is based on the idea from \cite[chapter 2]{DGV}.
\begin{proposition}\label{expansion in time} Let $u$ be a positive, weak supersolution of \eqref{DDP}-\eqref{mp} in an open set which compactly contains $E_T$. Suppose that for some $(y,s)\in E_T$ and some $\rho>0$ there holds
\begin{equation}\label{epansion assumption}\left|[u(\cdot,s)>M]\cap K_\rho(y)\right|\geq \mu |K_\rho(y)|\end{equation} for some $M>0$
and $\mu\in(0,1)$. There exist constants $\eta$, $\delta$ in $(0,1)$ and $\gamma$, $C>1$ depending only upon $p$, $m$, $N$, $C_0$, $C_1$ and $\mu$, such that
\begin{equation}\label{time expansion formula}u(\cdot,t)\geq \eta M\quad\mathrm{a.\ e.}\quad\mathrm{in}\quad K_{2\rho}(y)\times\left(s+\frac{C^{p+m-3}}{(\eta M)^{p+m-3}}\frac{1}{2}\delta \rho^p,s+\frac{C^{p+m-3}}{(\eta M)^{p+m-3}}\delta \rho^p\right].\end{equation}
Moreover, the functional dependence of $\eta$ on the parameter $\mu$ is of the form
\begin{equation}\label{parameter dependence 2}\eta=\eta_0\mu B^{-\mu^{-d}}\end{equation}
for constants $\eta_0$, $B$, $d$ depending only upon the data $p$, $m$, $C_0$ and $C_1$.
\end{proposition}
\begin{proof} Without loss of generality, we may prove the proposition in the case that $(y,s)=(0,0)$. This is always possible
by using a translation. In the following, we divide our proof in three steps.

\emph{Step 1. Changing the time variable.}
Observe that assumption \eqref{epansion assumption} implies
\begin{equation*}\left|[u(\cdot,0)>\sigma M]\cap K_\rho\right|\geq \mu |K_\rho|\end{equation*}
for all $\sigma\leq1$.
We now use Lemma \ref{K>} with $M$ replaced by $\sigma M$ and conclude that, with $\kappa$, $\delta$ defined in \eqref{parameter1}, there holds
\begin{equation}\label{u ex}\left|\left[u\left(\cdot,(\sigma M)^{3-m-p}\delta \rho^p\right)>\sigma \kappa M\right]\cap K_\rho\right|\geq \frac{1}{2}\mu |K_\rho|\quad\mathrm{for\ \ all}\quad\sigma\leq1.\end{equation}
For $\tau\geq 0$, set
$$\sigma=e^{-\frac{\tau}{p+m-3}},\quad t(\tau)=\frac{e^\tau}{M^{p+m-3}}\delta \rho^p\quad \mathrm{and}\quad \nu(\tau)=\frac{1}{M}e^{\frac{\tau}{p+m-3}}(\delta \rho^p)^{\frac{1}{p+m-3}}.$$
At this stage, we set $v(x,\tau)=\nu(\tau)u(x,t(\tau))$.
Notice from \eqref{u ex} that
\begin{equation}\label{U ex}\left|\left[v\left(\cdot, \tau\right)>M_0\right]\cap K_\rho\right|\geq \frac{1}{2}\mu |K_\rho|\end{equation}
for all $\tau\geq0$, where $M_0:=\kappa (\delta\rho^p)^{1/(p+m-3)}$.
Denote by $\tilde A:E\times\mathbb R^+\times \mathbb R^{N+1}\to \mathbb R^N$ the vector field:
$$\tilde A(x,\tau,v,Dv)=\nu(\tau)^{p+m-2}A\left[x,t(\tau),u(x,t(\tau)),Du(x,t(\tau))\right].$$
Then the function $v$ is a weak supersolution to $\partial_\tau v-\mathrm{div}\tilde A(x,\tau,v,Dv)=0,$
i.e., the inequality
\begin{equation*}\label{weak form}
-\int_0^{+\infty} \int_Ev\frac{\partial \Phi }{\partial \tau}dxd\tau+ \int_0^{+\infty} \int_E\tilde A(x,\tau,v,Dv)\cdot D\Phi dxd\tau\geq 0\end{equation*}
holds for any non-negative test function $\Phi\in C_0^\infty[E\times (0,+\infty)]$. The structure conditions for $\tilde A(x,\tau,v,Dv)$ are as follows:
\begin{equation*}\label{structure}
	\begin{cases}
	\tilde A(x,\tau,v,Dv)\cdot Dv\geq \widetilde{C_0}|v|^{m-1}|Dv|^p,\\
	|\tilde A(x,\tau,v,Dv)|\leq \widetilde{C_1}|v|^{m-1}|Dv|^{p-1},
	\end{cases}
\end{equation*}
for some constants $\widetilde{C_0}$ and $\widetilde{C_1}$ depending only upon $m$, $p$, $C_0$ and $C_1$. For the proof, see for example \cite{FS}.

\emph{Step 2. Expanding the positivity of $v$.}
Let $\zeta$ be a piecewise smooth cut-off function vanishes on the parabolic boundary $Q_{16\rho}^+(\theta)$ such that
\begin{equation*}\zeta=1\quad\mathrm{in}\quad\mathcal Q_{8\rho}(\theta):=K_{8\rho}\times(\theta(8\rho)^p,\theta(16\rho)^p],\quad |D\zeta|\leq\frac{1}{8\rho}\quad\mathrm{and}\quad|\partial_t\zeta|\leq\frac{1}{\theta(8\rho)^p}.\end{equation*}
In the case $m>1$ and $p> 2$, we use the Caccioppoi estimate \eqref{CC} for $(v-k)_-$ and the test function $\zeta$ defined above to infer
\begin{equation}\begin{split}\label{gradient1}\iint_{\mathcal Q_{8\rho}(\theta)}&v^{m-1}|D(v-k)_-|^pdxd\tau
\\ &\leq \tfrac{1}{2}\iint_{Q_{16\rho}^+(\theta)}(v-k)_-^2\left|\frac{\partial \zeta^p}{\partial \tau}\right|dxd\tau+\gamma \iint_{Q_{16\rho}^+(\theta)}v^{m-1}(v-k)_-^p|D\zeta|^pdxd\tau
\\ &\leq \gamma\left(\frac{k^2}{\theta(8\rho)^p}
+\frac{k^{p+m-1}}{(8\rho)^p}\right)|\mathcal Q_{8\rho}(\theta)|.\end{split}\end{equation}
In the case $m+p>3$ and $0<m<1$, we apply the Caccioppoi estimate \eqref{CCC} for $(v-k)_-$ with the same cut-off function $\zeta$ as above to get a gradient estimate
\begin{equation}\begin{split}\label{gradient2}\iint_{\mathcal Q_{8\rho}(\theta)}& |D(v-k)_-|^pdxd\tau
\\ \leq & \gamma k^{2-m}\iint_{Q_{16\rho}^+(\theta)}(v-k)_-\left|\frac{\partial \zeta}{\partial \tau}\right|dxd\tau+\gamma k^p\iint_{Q_{16\rho}^+(\theta)}\chi_{[v\leq k]}|D\zeta|^pdxd\tau
\\ \leq & \gamma\frac{k^p}{(8\rho)^p}\left(1+\frac{1}{\theta k^{p+m-3}}\right)|\mathcal Q_{8\rho}(\theta)|.\end{split}\end{equation}
At this point, we introduce the levels $k_j=2^{-j}M_0$ with $j=0,1,\cdots j_*$ where $j_*$ is to be taken depending on $\mu$, $m$, $p$, $C_0$ and $C_1$. In the gradient estimates \eqref{gradient1} and \eqref{gradient2}, we choose $k=k_j$ and
\begin{equation}\label{theta}\theta=k_{j_*}^{3-p-m}=(2^{j_*}M_0^{-1})^{p+m-3}.\end{equation}
If we set $v_j=\max\{k_{j+2},v\}$, then the estimates \eqref{gradient1} and \eqref{gradient2} imply that
\begin{equation}\label{gradient3}\iint_{\mathcal Q_{8\rho}(\theta)}|D(v_j-k_j)_-|^pdxd\tau\leq\gamma\frac{k_j^p}{(8\rho)^p}|\mathcal Q_{8\rho}(\theta)|.\end{equation}
To proceed further, we set
\begin{equation*}A_j(\tau)=\{v_j(\cdot,\tau)<k_j\}\cap K_{8\rho}\quad \mathrm{and}\quad A_j=\{v_j<k_j\}\cap \mathcal Q_{8\rho}(\theta).\end{equation*}
We observe from the definition of $v_j$ that $v_j=v$ on the set $A_j(\tau)\backslash A_{j+1}(\tau)$.
Then we use Lemma 2.2 in \cite[chapter I]{Di} and \eqref{U ex} to infer that
\begin{equation}\label{AD}\frac{1}{2}k_{j}|A_{j+1}(\tau)|\leq \gamma\mu^{-1}\rho\int_{[k_{j+2}< v(\cdot,\tau)<k_j]\cap K_{8\rho}}|Dv(\cdot,\tau)|dx.\end{equation}
Integrating both sides of \eqref{AD} with respect to $\tau$ over $(\theta(8\rho)^p,\theta(16\rho)^p)$, using H\"older's inequality and taking into account \eqref{gradient3}, we conclude that
\begin{equation*}\begin{split}\frac{k_{j}}{2}|A_{j+1}|&\leq \frac{\gamma\rho}{\mu}\left(\iint_{\mathcal Q_{8\rho}(\theta)}
|D(v-k_j)_-|^pdxd\tau\right)^{\frac{1}{p}}|A_j\backslash A_{j+1}|^{\frac{p-1}{p}}\leq\frac{\gamma k_j}{\mu}|\mathcal Q_{8\rho}(\theta)|^{\frac{1}{p}}|A_j\backslash A_{j+1}|^{\frac{p-1}{p}}.
\end{split}\end{equation*}
We now proceed as in the proof of Proposition 4.1 in \cite[chapter 2]{DGV}. From the estimate above, we obtain
\begin{equation}\label{A}\left|\{v<k_{j_*}\}\cap \mathcal Q_{8\rho}(\theta)\right|=|A_{j_*}|\leq \frac{1}{\mu}\left(\frac{\gamma}{j_*}\right)^{\frac{p-1}{p}}|\mathcal Q_{8\rho}(\theta)|,\end{equation}
where $j_*$ is to be chosen.
At this stage, we are ready to invoke Lemma \ref{Q>}. With the choice of $\theta$ in \eqref{theta}, we obtain from \eqref{nu-m<1} that $\nu_-=\gamma^{-\frac{N}{p}}
2^{-2N-3-\frac{3N}{p}-\frac{2N^2}{p}}$. We now choose $j_*$ be an integer depending only upon $m$, $p$, $N$, $C_0$ and $C_1$, such that
\begin{equation*}\label{j*}j_*\geq\gamma\left(\tfrac{1}{2}\mu\nu_-\right)^{-\frac{p}{p-1}}.\end{equation*}
Then we apply Lemma \ref{Q>} to $v$
with $M=k_{j_*}$ and $a=\frac{1}{2}$ over the parabolic cylinder
$\mathcal Q_{8\rho}(\theta)=(0,\tau_*)+Q_{8\rho}^-(\theta)$ where $\tau_*=\theta (16\rho)^p$,
and conclude that
\begin{equation}\label{v>}v(x,\tau)\geq \tfrac{1}{128}2^{-j_*}\mu(\delta\rho^p)^{\frac{1}{p+m-3}}\quad\mathrm{a.\ e.}\quad\mathrm{in}\quad
(0,\tau_*)+Q_{4\rho}^-(\theta).\quad\end{equation}

\emph{Step 3. Expanding the positivity of $u$.} We shall adopt the same procedure as in the proof of Proposition 4.1 in \cite{FS}.
As $\tau$ ranges over the interval $(\theta(16\rho)^p-\theta(4\rho)^p,\theta(16\rho)^p)$, the function $e^{\frac{\tau}{p+m-3}}$ ranges over
\begin{equation*}b_1:=\exp\left\{\frac{(16^p-4^p)2^{j_*(p+m-3)}}{(p+m-3)\delta\mu^{p+m-3}}\right\}\leq e^{\frac{\tau}{p+m-3}}\leq
\exp\left\{\frac{16^p2^{j_*(p+m-3)}}{(p+m-3)\delta\mu^{p+m-3}}\right\}=:b_2.\end{equation*}
Let $\eta=(128 b_2)^{-1}2^{-j_*}\mu$. Transforming \eqref{v>} back to the time variable $t$, we have
\begin{equation*}u(x,t)\geq \eta M\quad\mathrm{a.\ e.}\quad\mathrm{in}\quad K_{2\rho}\times\left(\frac{C^{p+m-3}}{(\eta M)^{p+m-3}}\frac{1}{2}\delta \rho^p,
\frac{C^{p+m-3}}{(\eta M)^{p+m-3}}\delta \rho^p\right]\end{equation*}
for a suitable $C$ depending only upon $p$, $m$, $C_0$, $C_1$ and $\mu$. We have thus proved the proposition.
\end{proof}
\section{The Hot Alternative}\label{Hot alternative}
The proof of Theorem \ref{main result1} can be reduced by considering two alternatives. In this section, we shall discuss a situation which will be called the hot alternative.
Before proceeding further, we remark that the exponential decay of $\eta$ on parameter $\mu$ in Proposition \ref{expansion in time} need to be improved.
The next Lemma asserts that the functional dependence \eqref{parameter dependence 2} can be sharpened to $\eta_*=\eta_0\mu^{d+1}$. Our proof follows the scheme of \cite[chapter 3, Proposition 7.1]{DGV}, which makes no use of covering arguments.
\begin{lemma}\label{refined} Let $u$ be a positive, weak supersolution of \eqref{DDP}-\eqref{mp} in an open set which compactly contains $E_T$. Suppose that for some $(y,s)\in E_T$ and some $\rho>0$ there holds
\begin{equation}\label{k>>}\left|[u(\cdot,s)\geq M]\cap K_\rho(y)\right|\geq \mu|K_\rho(y)|\end{equation}
for some $M>0$ and some $\mu\in(0,1)$. There exist constants $\eta_0$, $\delta$ in $(0,1)$ and $C$, $d>0$ depending only upon $m$, $p$, $N$, $C_0$ and $C_1$, such that
\begin{equation}u(\cdot,t)\geq \eta_*M\quad\mathrm{a.\ e.}\quad\mathrm{in}\quad K_{2\rho}(y)\times\left(s+\frac{C^{p+m-3}}{(\eta_*M)^{p+m-3}}\frac{1}{2}\delta\rho^p,s+\frac{C^{p+m-3}}{(\eta_*M)^{p+m-3}}\delta\rho^p\right],\end{equation}
where $\eta_*=\eta_0\mu^{d+1}$.
\end{lemma}
\begin{proof} It suffices to prove the lemma in the case that $(y,s)=(0,0)$. This is always possible by using a translation. By Lemma \ref{K>}, there exist $\delta$ and $\kappa$
defined in \eqref{parameter1}, such that
\begin{equation}\label{kkk}\left|[u(\cdot,t)\geq \kappa M]\cap K_\rho(y)\right|\geq \frac{1}{2}\mu|K_\rho(y)|\end{equation}
for almost every $t\in\left(0,\delta M^{3-m-p}\rho^p\right]$. In the following, the proof will be split into two steps.

\emph{Step 1. The estimate \eqref{kkk} revisited.} We claim that there exists a time level $t_1$ in the range
\begin{equation}\label{t_1}M^{3-m-p}\frac{1}{2}\delta\rho^p<t_1\leq M^{3-m-p}\delta\rho^p,\end{equation}
and $x_0\in K_\rho$
such that, with $\eta_*=\eta_1\mu^{d_1}$,
\begin{equation}\label{ut1}\left|[u(\cdot,t_1)>\tfrac{1}{2}\kappa M]\cap K_{\eta_*\rho}(x_0)\right|>\frac{1}{2}|K_{\eta_* \rho}(x_0)|\end{equation}
where $d_1$, $\eta_1$ are the constants depending only upon $m$, $p$, $C_0$ and $C_1$.

In order to prove this assertion, we shall use Caccioppoli estimates in Proposition \ref{C1} and Proposition \ref{C2} for $(u-M)_-$ with the choice of a cut-off function $\zeta$ such that
$$\zeta=1\quad\mathrm{in}\quad\mathcal Q,\quad\zeta=0\quad\mathrm{on}\quad\partial_P\mathcal Q^\prime,$$
and
$$0\leq \frac{\partial \zeta}{\partial t}\leq \frac{4^p}{\delta M^{3-m-p}\rho^p},\quad|D\zeta|\leq\frac{4}{\rho},$$
where
$$\mathcal Q:=K_\rho\times \left(\frac{1}{2}\frac{\delta}{M^{p+m-3}}\rho^p,\frac{\delta}{M^{p+m-3}}\rho^p\right],\quad \mathcal Q^\prime:=K_{2\rho}\times \left(0,\frac{\delta}{M^{p+m-3}}\rho^p\right].$$
In the case $p> 2$ and $m>1$, we apply Caccioppoli estimate \eqref{CC} for $(u-M)_-$ with the cut-off function $\zeta$ defined above to obtain
\begin{equation*}\begin{split}\iint_{\mathcal Q}u^{m-1}|D(u-M)_-|^pdz\leq&\iint_{\mathcal Q^\prime}u^{m-1}|D(u-M)_-|^p\zeta^pdz\\ \leq &\iint_{\mathcal Q^\prime}(u-M)_-^2\left|\frac{\partial \zeta^p}{\partial t}\right|dz+2\gamma \iint_{\mathcal Q^\prime}u^{m-1}(u-M)_-^p|D\zeta|^pdz\\
\leq &\gamma \frac{1}{\delta M^{1-m-p}\rho^p}|\mathcal Q|.\end{split}\end{equation*}
If we introduce a cut-off function $v=\max\left\{u,\frac{1}{4}\kappa M\right\}$, then we obtain the estimate
\begin{equation*}\begin{split}\left(\frac{\kappa M}{4}\right)^{m-1}\iint_{\mathcal Q}|D(v-M)_-|^pdz&\leq \iint_{\mathcal Q}v^{m-1}|D(v-M)_-|^pdz
\\&=\iint_{\mathcal Q\cap [v=u]}v^{m-1}|D(v-M)_-|^pdz+\iint_{\mathcal Q\cap [v>u]}v^{m-1}|D(v-M)_-|^pdz\\&\leq \iint_{\mathcal Q}u^{m-1}|D(u-M)_-|^pdz\\&\leq \gamma \frac{1}{\delta M^{1-m-p}\rho^p}|\mathcal Q|,\end{split}\end{equation*}
since $m>1$. Recalling from \eqref{parameter1}, we rewrite the above estimate as follows:
\begin{equation}\begin{split}\label{Dv}\iint_{\mathcal Q}|D(v-M)_-|^pdz\leq \frac{\gamma M^p}{\delta\mu^{m-1}\rho^p}|\mathcal Q|.\end{split}\end{equation}
From the definition of $v$ and \eqref{kkk}, we conclude that
\begin{equation}\label{kkkk}\left|[v(\cdot,t)\geq \kappa M]\cap K_\rho(y)\right|\geq \frac{1}{2}\mu|K_\rho(y)|\end{equation}
for almost every $t\in\left(0,\delta M^{3-m-p}\rho^p\right]$. To proceed further, we introduce the change of variables
$$y=\frac{x}{\rho},\quad \tau=\frac{t}{\delta M^{3-m-p}\rho^p}\quad\mathrm{and}\quad \tilde v=\frac{(M-v)_+}{M}.$$
It follows from \eqref{parameter1} and \eqref{Dv} that
\begin{equation}\begin{split}\label{Dvv}\int_{\frac{1}{2}}^1\int_{K_1}|D\tilde v(y,\tau)|^pdyd\tau\leq \frac{\gamma}{\mu^{m+p}}.\end{split}\end{equation}
Moreover, we rewrite the  estimate \eqref{kkkk} in terms of $\tilde v$ as follows:
\begin{equation}\label{kkkkk}\left|[\tilde v(\cdot,\tau)<1- \kappa ]\cap K_1\right|\geq \frac{1}{2}\mu\quad\mathrm{for}\ \mathrm{all}\quad\tau\in\left(\tfrac{1}{2},1\right].\end{equation}
From \eqref{Dvv} and \eqref{kkkkk}, there exists a time level $\tau_1\in\left(\frac{1}{2},1\right]$ such that
\begin{equation*}\int_{K_1}|D\tilde v(\cdot,\tau_1)|dy\leq \frac{\gamma}{\mu^{\frac{m+p}{p}}}\quad\mathrm{and}\quad \left|[\tilde v(\cdot,\tau_1)<1- \kappa ]\cap K_1\right|\geq \frac{\mu}{2}.\end{equation*}
At this point, we infer that there exists a $y_0\in K_1$ such that,
\begin{equation}\begin{split}\label{eta8}\left|[\tilde v(\cdot,\tau_1)<1- \tfrac{1}{2}\kappa ]\cap K_{\eta_*}(y_0)\right|\geq \frac{1}{2}|K_{\eta_*}(y_0)|,\quad\mathrm{where}\quad \eta_*=\eta_1\mu^{d_1},\end{split}\end{equation}
and $d_1$, $\eta_1$ depending only upon $m$ and $p$ (see for instance \cite{localclustering}). Transforming the above estimate back to the original variables $x$, $t$ and original function $v$, we see that there exists a time level $t_1$ in the range \eqref{t_1} and
\begin{equation}\label{vt1}\left|[v(\cdot,t_1)>\tfrac{1}{2}\kappa M]\cap K_{\eta_*\rho}(x_0)\right|>\frac{1}{2}|K_{\eta_* \rho}(x_0)|.\end{equation}
Recalling from $v=\max\left\{u,\frac{1}{4}\kappa M\right\}$ that
$$\left[v(\cdot,t_1)>\tfrac{1}{2}\kappa M\right]\cap K_{\eta_*\rho}(x_0)=\left[u(\cdot,t_1)>\tfrac{1}{2}\kappa M\right]\cap K_{\eta_*\rho}(x_0).$$
Then the estimate \eqref{vt1} implies \eqref{ut1}, which proves the claim for the case $p>2$ and $m>1$.

We now turn our attention to the case $p+m>3$ and $0<m<1$.
We use Caccioppoli estimate \eqref{CCC} for $(u-M)_-$ with the same cut-off function $\zeta$ to obtain
\begin{equation*}\begin{split}M^{m-1}\iint_{\mathcal Q^\prime}|D(u-M)_-|^p\zeta^pdz
&\leq\gamma M\iint_{\mathcal Q^\prime}(u-M)_-\left|\frac{\partial \zeta}{\partial t}\right|dz+\gamma M^{m+p-1}\iint_{\mathcal Q^\prime}\chi_{[u\leq M]}|D\zeta|^pdz\\&\leq \gamma\frac{M^{p+m-1}}{\delta\rho^p}|\mathcal Q|,\end{split}\end{equation*}
which yields
\begin{equation}\begin{split}\label{Duu}\iint_{\mathcal Q}|D(u-M)_-|^pdz\leq \gamma\frac{M^p}{\delta\rho^p}|\mathcal Q|.\end{split}\end{equation}
Observe that \eqref{Duu} is even better than \eqref{Dv} and the claim follows by the same method as in the previous case.

\emph{Step 2. Iterative arguments.} The proof of this part follows in a same manner as the proof of \cite[chapter 3, Proposition 7.1]{DGV} and we just sketch the proof.
By Proposition \ref{expansion in time}, there exist $\tilde\eta$ and $\tilde\delta$ in $(0,1)$ and $C$ depending only upon the data $p$, $m$, $C_0$ and $C_1$, such that
\begin{equation}\label{start}u(\cdot,t_2)\geq \frac{1}{2}\kappa\tilde \eta M\quad\mathrm{a.\ e.}\quad\mathrm{in}\quad K_{2\eta_*\rho}(x_0),\end{equation}
for all times
\begin{equation*}t_1+\frac{C^{p+m-3}}{(\kappa\tilde \eta M)^{p+m-3}}\frac{1}{2}\tilde\delta(\eta_*\rho)^p\leq t_2\leq t_1+\frac{C^{p+m-3}}{(\kappa\tilde \eta M)^{p+m-3}}\tilde\delta(\eta_*\rho)^p.\end{equation*}
At this stage, we use Proposition \ref{expansion in time} again from \eqref{start} with $M$ replaced by $\frac{1}{2}\kappa\tilde \eta M$ and $\mu=1$. This implies that
there exist $\bar\eta$ and $\bar\delta$ in $(0,1)$ and $\bar C$ depending only upon the data $p$, $m$, $C_0$, $C_1$, such that
\begin{equation}\label{start1}u(\cdot,t_3)\geq \frac{1}{2}\kappa\bar \eta\tilde\eta M,\end{equation}
for almost every $x\in K_{4\eta_*\rho}(x_0)$ and all times
\begin{equation}\label{start time}t_2+\frac{\bar C^{p+m-3}}{(\kappa\bar\eta\tilde \eta M)^{p+m-3}}\frac{1}{2}\bar\delta(2\eta_*\rho)^p\leq t_3\leq t_1+\frac{\bar C^{p+m-3}}{(\kappa\bar \eta \tilde\eta M)^{p+m-3}}\bar\delta(2\eta_*\rho)^p.\end{equation}
From \eqref{start1} and \eqref{start time}, repeated application of Proposition \ref{expansion in time} with $\mu=1$ enable us to deduce that
\begin{equation*}u(\cdot,t)\geq \frac{1}{2}\kappa\bar \eta^n\tilde\eta M\quad\mathrm{a.\ e.}\quad\mathrm{in}\quad K_{2^{n+2}\eta_*\rho}(x_0),\end{equation*}
for any time level $t$ satisfies
\begin{equation*}t_{n+1}+\frac{\bar C^{p+m-3}}{(\kappa\bar\eta^n\tilde \eta M)^{p+m-3}}\frac{1}{2}\bar\delta(2^{n+2}\eta_*\rho)^p\leq t\leq t_{n+1}+\frac{\bar C^{p+m-3}}{(\kappa\bar \eta^n \tilde\eta M)^{p+m-3}}\bar\delta(2^{n+2}\eta_*\rho)^p.\end{equation*}
The iteration stops when $2^{n+2}\eta_*=2$ and the estimate \eqref{k>>} follows with the choice
$$d=-d_1\log_2\bar\eta\quad\mathrm{and}\quad\eta_0=128^{-1}\bar \eta^{-2}(\frac{2}{\eta_1})^{\log_2\bar\eta}\tilde\eta.$$
The proof of the lemma is now complete.
\end{proof}
We now state the main result of this section. Here and subsequently, the condition \eqref{assumption1} introduced in the following proposition is referred as \emph{hot alternative}.
\begin{proposition}\label{hot pro} Let $d$, $\eta_0$ and $C$ be as in Lemma \ref{refined}. Suppose that $u$ is a positive, weak super-solution in an open set which compactly contains
$K_2\times(0,S_1]$ where
\begin{equation}\label{T1 time}S_1=\delta\left(\frac{C}{\eta_0}\right)^{p+m-3}.\end{equation}
If there exist a time level $t_0\in(0,1]$ and a $k>1$,
such that
\begin{equation}\label{assumption1}|[u(\cdot,t_0)>k^{1+\frac{1}{d}}\cap K_1]|>k^{-\frac{1}{d}}|K_1|,\end{equation}
then
\begin{equation}\label{result11}u\geq \eta_0\quad\mathrm{a.\ e.\ \ in}\quad K_2\times\left(1+\tfrac{1}{2}S_1,S_1\right].\end{equation}
\end{proposition}
\begin{proof}
By reducing $\eta_1$ in \eqref{eta8}, we deduce $1+\tfrac{1}{2}C^{m+p-3}\eta_0^{-(p+m-3)}\delta\leq C^{m+p-3}\eta_0^{-(p+m-3)}\delta$.
The estimate \eqref{result11} follows immediately from Lemma \ref{refined} with $\rho=1$, $M=k^{1+\frac{1}{d}}$ and $\mu=k^{-\frac{1}{d}}$. This completes the proof.
\end{proof}
\section{The Cold Alternative}\label{Cold alternative}
In this section, we consider the case complement to the hot alternative.
We shall prove an analogue result of Proposition \ref{hot pro} in this case.
The proof is based on a suitable modifications
of Moser's iteration technique \cite{moser} following the adaptation of the Kuusi's technique \cite{Kuusi} to doubly degenerate parabolic equations with ideas from \cite{Le}.
To start with, we prove the following higher integrability result which plays a crucial role in the proof of the uniform estimates.
\begin{lemma} \label{moser integral} Let $u$ be a positive, weak supersolution in an open set which compactly contains
$K_1\times(0,1]$. There exists a number $\nu(p,N)$ depending only upon $p$ and $N$, with $0<\nu(p,N)<1$ such that for all $q$ in the range
$$p+m-3<q<p+m-2+\nu(p,N)$$
and all $s$ given by
$$s=p+m-3+\left(1+\nu(p,N)\right)^{-(n+1)}(q-(p+m-3)),\quad n=1,2\cdots$$
the following holds: if there exists a constant $\tilde C$ depending only upon $m$, $p$, $N$, $C_0$ and $C_1$, such that
$$\int_0^1\int_{K_1}u^sdxdt\leq \tilde C$$
then
\begin{equation}\label{C}\int_0^{a^p}\int_{K_a}u^qdxdt\leq C\end{equation}
for all $a\in(\frac{1}{2},\frac{8}{9})$ and the constant $C$ depending only upon $m$, $p$, $N$, $C_0$, $C_1$, $s$ and $q$.
\end{lemma}
\begin{proof}
Throughout the proof, we set
$$\rho_0=1,\quad\rho_j=1-(1-a)\frac{1-2^{-j}}{1-2^{-n-1}},\quad K_j=K_{\rho_j}\quad\mathrm{and}\quad Q_j=K_{\rho_j}\times(0,\rho_j^p]$$
where $j=0,1,\cdots,n+1$. Moreover, we choose non-negative, smooth functions $\varphi_j$ with support lying in $Q_j$,
such that
$$\varphi_j=1\quad\mathrm{on}\quad Q_{j+1},\quad\varphi_j=0\quad\mathrm{on}\quad\partial Q_j\backslash \left(K_j\times\{0\}\right),$$
$$|D\varphi_j|\leq \frac{2^{j+1}}{1-a}\quad\mathrm{and}\quad0\leq-\frac{\partial \varphi_j}{\partial t}\leq \frac{2^{p(j+1)}}{(1-a)^p},\quad\mathrm{where}\quad j=0,1,\cdots,n.$$
Let $\epsilon\in(-1,0)$ be as in Proposition \ref{C2}. At this stage, we want to apply parabolic Sobolev embedding estimates from Proposition \ref{m<1Sobolev}.

In the case $p<N$, set
$$\sigma=p+m-2+\epsilon,\quad\kappa=1+\frac{p(1+\epsilon)}{N(p+m-2+\epsilon)},\quad\beta=\frac{p(p+m-2+\epsilon)}{1+\epsilon}.$$
We point out that $(\kappa-1)N$ may less than one provided $p$ is large and we cannot use the inequality \eqref{Sobolev formula} for $r=(\kappa-1)N$.
Instead, we use parabolic Sobolev inequality \eqref{p<N r<1} to obtain
\begin{equation}\begin{split}\label{sobolev cold1}\iint_{Q_{j+1}}u^{\kappa\sigma}dz&\leq\iint_{Q_j}(u^{\frac{\sigma}{p}}\varphi_j^{\frac{\beta}{p}})^{\kappa p}dz\\&\leq \gamma\iint_{Q_j}|D(u^{\frac{\sigma}{p}}\varphi_j^{\frac{\beta}{p}})|^pdz\left(\esssup_{0<t<\rho_j^p}\int_{K_j}(u^{\frac{\sigma}{p}}\varphi_j^{\frac{\beta}{p}})
^{(\kappa-1)N}dx\right)^{\frac{p}{N}}\\&=\gamma\iint_{Q_j}|D(u^{\frac{\sigma}{p}}\varphi_j^{\frac{\beta}{p}})|^pdz\left(\esssup_{0<t<\rho_j^p}
\int_{K_j}u^{1+\epsilon}\varphi_j^pdx\right)^{\frac{p}{N}}.\end{split}\end{equation}

In the case $p>N$, set
$$\sigma=p+m-2+\epsilon,\quad\kappa=1+\frac{N(1+\epsilon)}{p(p+m-2+\epsilon)},\quad\beta=\frac{p(p+m-2+\epsilon)}{1+\epsilon}.$$
Using the parabolic Sobolev inequality \eqref{p>N r<1} with
$$v=u^{\frac{\sigma}{p}}\varphi_j^{\frac{\beta}{p}},\quad q=\frac{p}{N}\quad\mathrm{and}\quad r=\frac{p(1+\epsilon)}{p+m-2+\epsilon},$$
we conclude that
\begin{equation}\begin{split}\label{sobolev cold2}\iint_{Q_{j+1}}u^{\kappa\sigma}dz&\leq\iint_{Q_j}(u^{\frac{\sigma}{p}}\varphi_j^{\frac{\beta}{p}})^{p+\frac{r}{q}}dz\\&\leq \gamma\iint_{Q_j}|D(u^{\frac{\sigma}{p}}\varphi_j^{\frac{\beta}{p}})|^pdz\left(\esssup_{0<t<\rho_j^p}\int_{K_j}(u^{\frac{\sigma}{p}}\varphi_j^{\frac{\beta}{p}})
^{\frac{p(1+\epsilon)}{p+m-2+\epsilon}}dx\right)^{\frac{N}{p}}\\&=\gamma\iint_{Q_j}|D(u^{\frac{\sigma}{p}}\varphi_j^{\frac{\beta}{p}})|^pdz\left(\esssup_{0<t<\rho_j^p}
\int_{K_j}u^{1+\epsilon}\varphi_j^pdx\right)^{\frac{N}{p}}.\end{split}\end{equation}

In the case $p=N$, set
$$\sigma=N+m-2+\epsilon,\quad\kappa=1+\frac{1+\epsilon}{2(N+m-2+\epsilon)},\quad\beta=\frac{N(N+m-2+\epsilon)}{1+\epsilon}.$$
Applying the parabolic embedding estimate \eqref{p=N r<1} with
$$v=u^{\frac{\sigma}{N}}\varphi_j^{\frac{\beta}{N}}\quad\mathrm{and}\quad r=\frac{N(1+\epsilon)}{N+m-2+\epsilon},$$
we deduce that
\begin{equation}\begin{split}\label{sobolev cold3}\iint_{Q_{j+1}}u^{\kappa\sigma}dz&\leq\iint_{Q_j}(u^{\frac{\sigma}{N}}\varphi_j^{\frac{\beta}{N}})^{N+\frac{r}{2}}dz\\&\leq \gamma\iint_{Q_j}|D(u^{\frac{\sigma}{N}}\varphi_j^{\frac{\beta}{N}})|^pdz\left(\esssup_{0<t<\rho_j^N}\int_{K_j}(u^{\frac{\sigma}{N}}\varphi_j^{\frac{\beta}{N}})
^{\frac{N(1+\epsilon)}{N+m-2+\epsilon}}dx\right)^{\frac{1}{2}}\\&=\gamma\iint_{Q_j}|D(u^{\frac{\sigma}{N}}\varphi_j^{\frac{\beta}{N}})|^pdz\left(\esssup_{0<t<\rho_j^N}
\int_{K_j}u^{1+\epsilon}\varphi_j^Ndx\right)^{\frac{1}{2}}.\end{split}\end{equation}
Combining estimates \eqref{sobolev cold1}-\eqref{sobolev cold3} and set \begin{equation}\label{nupN}\nu(p,N)=\begin{cases}\ \frac{p}{N},&p<N,\\ \ \frac{1}{2},&p=N,\\ \ \frac{N}{p},&p>N,\end{cases}\end{equation}
we conclude that in either case, with
$$\sigma=p+m-2+\epsilon,\quad\kappa=1+\nu(p,N)\frac{1+\epsilon}{p+m-2+\epsilon},\quad\beta=\frac{p(p+m-2+\epsilon)}{1+\epsilon},$$
there holds
\begin{equation}\begin{split}\label{sobolev cold}\iint_{Q_{j+1}}&u^{\kappa\sigma}dz\leq \gamma\iint_{Q_j}|D(u^{\frac{\sigma}{p}}\varphi_j^{\frac{\beta}{p}})|^pdz\left(\esssup_{0<t<\rho_j^p}
\int_{K_j}u^{1+\epsilon}\varphi_j^pdx\right)^{\nu(p,N)}.\end{split}\end{equation}
We now apply Caccioppoli estimate \eqref{moser2} to obtain
\begin{equation*}\begin{split}\esssup_{0<t<\rho_j^p}&\int_{K_j}u^{1+\epsilon}\varphi_j^pdx\leq
\frac{\gamma}{|\epsilon|^p(1+\epsilon)^p}\left(\iint_{Q_j}u^{m+p+\epsilon-2}|D\varphi_j|^pdz+\iint_{Q_j}u^{1+\epsilon}\left|\frac{\partial \varphi_j}{\partial t}\right|dz\right)\end{split}\end{equation*}
and
\begin{equation*}\begin{split}\iint_{Q_j}&|D(u^{\frac{\sigma}{p}}\varphi_j^{\frac{\beta}{p}})|^pdz\\&\leq\gamma\left(\iint_{Q_j}u^{m-2+\epsilon}|Du|^p\varphi_j^pdz
+\frac{1}{(1+\epsilon)^p}\iint_{Q_j}u^{m+p+\epsilon-2}|D\varphi_j|^pdz\right)\\&\leq
\frac{\gamma}{|\epsilon|^p(1+\epsilon)^p}\left(\iint_{Q_j}u^{m+p+\epsilon-2}|D\varphi_j|^pdz+\iint_{Q_j}u^{1+\epsilon}\left|\frac{\partial \varphi_j}{\partial t}\right|dz\right).\end{split}\end{equation*}
Then we conclude that
\begin{equation*}\begin{split}\iint_{Q_{j+1}}&u^{p+m-2+\epsilon+(1+\epsilon)\nu(p,N)}dz
\\&\leq \left[\frac{\gamma}{|\epsilon|^p(1+\epsilon)^p}\left(\iint_{Q_j}u^{m+p+\epsilon-2}|D\varphi_j|^pdz+\iint_{Q_j}u^{1+\epsilon}\left|\frac{\partial \varphi_j}{\partial t}\right|dz\right)\right]^{1+\nu(p,N)}\\&\leq \left[\frac{\gamma}{|\epsilon|^p(1+\epsilon)^p}\frac{2^{p(j+1)}}{(1-a)^p}\left(\iint_{Q_j}u^{p+m+\epsilon-2}dz+\iint_{Q_j}u^{1+\epsilon}dz
\right)\right]^{1+\nu(p,N)}.\end{split}\end{equation*}
Furthermore, we obtain from Young's inequality that
\begin{equation*}\begin{split}u^{1+\epsilon}&\leq u^{p+m+\epsilon-2}+\tfrac{p+m-3}{p+m+\epsilon-2}\left(\tfrac{p+m+\epsilon-2}{1+\epsilon}\right)^{-\frac{1+\epsilon}{p+m-3}}
\\&\leq u^{p+m+\epsilon-2}+\left(\tfrac{p+m-3}{1+\epsilon}\right)^{-\frac{1+\epsilon}{p+m-3}}
\\&\leq u^{p+m+\epsilon-2}+\max\left\{1,e^{-\frac{1}{p+m-3}\ln(p+m-3)}\right\},\end{split}\end{equation*}
which implies
\begin{equation*}\begin{split}\iint_{Q_{j+1}}&u^{p+m-2+\epsilon+(1+\epsilon)\nu(p,N)}dz
\leq \left[\frac{\gamma}{|\epsilon|^p(1+\epsilon)^p}\frac{2^{p(j+1)}}{(1-a)^p}\left(\iint_{Q_j}u^{p+m+\epsilon-2}dz+1
\right)\right]^{1+\nu(p,N)}.\end{split}\end{equation*}
To proceed further, we set
$$\bar h=1+\nu(p,N),\quad\epsilon_j=\bar h^j(\epsilon_0+1)-1,\quad \sigma_j=p+m-2+\epsilon_j,$$
where $-1<\epsilon_0<-1+h^{-n}$.
These choices imply the estimate
\begin{equation*}\begin{split}\iint_{Q_{j+1}}&u^{\sigma_{j+1}}dz
\leq \left[\frac{\gamma}{|\epsilon_j|^p(1+\epsilon_j)^p}\frac{2^{p(j+1)}}{(1-a)^p}\left(\iint_{Q_j}u^{\sigma_j}dz+1
\right)\right]^{\bar h}\end{split}\end{equation*}
for $j=0,1,\cdots,n$. Choosing $\epsilon_0=-1+\bar h^{-1-n}(q-(p+m-3))$, we see that
$$\frac{1}{|\epsilon_j|^p(1+\epsilon_j)^p}\leq \frac{1}{|\epsilon_n|^p(1+\epsilon_0)^p}=:\bar\gamma(s,p,q,N)$$
and we finally have
\begin{equation*}\begin{split}\iint_{Q_{j+1}}&u^{\sigma_{j+1}}dz
\leq \left(\frac{\gamma\bar \gamma2^{p(j+1)}}{(1-a)^p}\iint_{Q_j}u^{\sigma_j}dz
\right)^{\bar h}+\left(\frac{\gamma\bar\gamma2^{pn}}{(1-a)^p}\right)^{\bar h}.\end{split}\end{equation*}
Since $\sigma_0=s$ and $\sigma_{n+1}=q$, then the estimate \eqref{C} follows by induction. The proof is completed.
\end{proof}
\begin{lemma}\label{uniform bound} Let $u$ be as in Lemma \ref{moser integral} and let $d$ be as in Lemma \ref{refined}. We assume that
\begin{equation}\label{cold assumption}|[u(\cdot,t)>k^{1+\frac{1}{d}}\cap K_1]|\leq k^{-\frac{1}{d}}|K_1|\end{equation}
for all $t\in (0,1]$ and all level $k\geq1$. Let $\nu(p,N)$ be as in \eqref{nupN}. Then for all $q$ with
\begin{equation}\label{range}p+m-3<q<p+m-2+\nu(p,N)\end{equation}
there exists a constant $\gamma$ depending only upon $m$, $p$, $N$, $C_0$, $C_1$ and $q$ such that
\begin{equation}\label{uuniform}\int_0^{\left(\frac{3}{4}\right)^p}\int_{K_{\frac{3}{4}}}u^qdxdt\leq \gamma.\end{equation}
\end{lemma}
\begin{proof} Set $\delta=2^{-1}(d+1)^{-1}$. We deduce from \eqref{cold assumption} that
\begin{equation}\begin{split}\label{uniform int}\int_{K_1}u^\delta(\cdot,t)dx\leq \frac{2}{\delta}|K_1|\quad \quad\mathrm{for}\quad\mathrm{a.\ e.}\quad t\in(0,1).\end{split}\end{equation}
Fix time levels $s$ and $r$ with $7/8<s<r<1$ and set $Q_s=K_s\times(0,1]$, $Q_r=K_r\times(0,1]$. Let $\varphi=\varphi(x)$ be a non-negative smooth function in $K_r$ such that $0\leq\varphi\leq1$, $\varphi=1$ in $K_s$ and $|D\varphi|\leq (r-s)^{-1}$. We now apply Proposition \ref{m<1Sobolev} and repeat the argument \eqref{sobolev cold1}-\eqref{sobolev cold} in the proof of Lemma \ref{moser integral} with $\epsilon$ replaced by $\delta-1$,
$$\sigma=p+m-3+\delta,\quad\kappa=1+\nu(p,N)\frac{\delta}{p+m-3+\delta},\quad\beta=\frac{p(p+m-3+\delta)}{\delta},$$
for the three cases ($p<N$, $p>N$, $p=N$) respectively. Then we conclude that in either case,
\begin{equation*}\begin{split}\iint_{Q_s}u^{p+m-3+(1+\nu(p,N))\delta}dz&\leq\iint_{Q_r}(u^{\frac{\sigma}{p}}\varphi_j^{\frac{\beta}{p}})^{p+\frac{\nu(p,N)\delta p}{\sigma}}dz\\&\leq\gamma\iint_{Q_r}|D(u^{\frac{\sigma}{p}}\varphi_j^{\frac{\beta}{p}})|^pdz\left(\esssup_{0<t<1}
\int_{K_1}u^{\delta}\varphi_j^pdx\right)^{\nu(p,N)}\\&\leq \gamma\iint_{Q_r}|D(u^{\frac{\sigma}{p}}\varphi_j^{\frac{\beta}{p}})|^pdz.\end{split}\end{equation*}
We are now in a position to use Caccioppoli estimate \eqref{moser1} with $\epsilon=\delta-1$,
\begin{equation*}\begin{split}\iint_{Q_r}|D(u^{\frac{\sigma}{p}}\varphi_j^{\frac{\beta}{p}})|^pdz&\leq\gamma\left(\esssup_{0<t<1}\int_{K_1}u^\delta(\cdot,t)dx
+\iint_{Q_r}u^{m+p-3+\delta}|D\varphi|^pdz\right)\\&\leq\gamma\left(1
+\iint_{Q_r}u^{m+p-3+\delta}|D\varphi|^pdz\right).\end{split}\end{equation*}
To estimate the right hand side from above, we use Young's inequality to obtain
\begin{equation*}\begin{split}\iint_{Q_r}&u^{m+p-3+\delta}|D\varphi|^pdz\leq \frac{1}{2\gamma}\iint_{Q_r}u^{m+p-3+\delta(1+\nu(p,N))}dz
+\gamma\left(\frac{1}{r-s}\right)^{\frac{N}{\delta}\left[p+m-3+\delta(1+\nu(p,N))\right]}.\end{split}\end{equation*}
Then we conclude that
\begin{equation*}\begin{split}\iint_{Q_s}&u^{p+m-3+(1+\nu(p,N))\delta}dz\leq \frac{1}{2}\iint_{Q_r}u^{m+p-3+\delta(1+\nu(p,N))}dz
+\gamma\left(\frac{1}{r-s}\right)^{\frac{N}{\delta}\left[p+m-3+\delta(1+\nu(p,N))\right]}.\end{split}\end{equation*}
At this stage, we deduce from \cite[Preliminaries Lemma 5.2]{DGV} that
$$\int_0^{(\frac{7}{8})^p}\int_{K_{\frac{7}{8}}}u^{p+m-3+(1+\nu(p,N))\delta}dxdt\leq\gamma,$$
and the estimate \eqref{uuniform} follows immediately from Lemma \ref{moser integral}. This completes the proof of Lemma \ref{uniform bound}.
\end{proof}
\begin{lemma}\label{uniform gradient} Let $u$ be as in Lemma \ref{uniform bound}. Then there exists a constant $\tilde \gamma$ depending only upon $m$, $p$, $N$, $C_0$ and $C_1$ such that
\begin{equation}\label{gradient uniform} \int_0^{\left(\frac{5}{8}\right)^p}\int_{K_{\frac{5}{8}}}u^{m-1}|Du|^{p-1}dxdt\leq \tilde \gamma.\end{equation}
\end{lemma}
\begin{proof} By H\"older's inequality
\begin{equation*}\begin{split}\iint_{Q_{\frac{5}{8}}}u^{m-1}|Du|^{p-1}dz
&=\iint_{Q_{\frac{5}{8}}}(u^{m-1}u^{-\frac{(p-1)(m-\delta-2)}{p}})(u^{\frac{(p-1)(m-\delta-2)}{p}}|Du|^{p-1})dz\\&\leq \left(\iint_{Q_{\frac{5}{8}}} u^{p+m-\delta-2+\delta p}dz\right)^{\frac{1}{p}}
\left(\iint_{Q_{\frac{5}{8}}} u^{m-\delta-2}|Du|^{p}dz\right)^{\frac{p-1}{p}},\end{split}\end{equation*}
where
$\delta=[4(p-1)]^{-1}\nu(p,N)<1.$  By Lemma \ref{uniform bound}, we see that
$$\iint_{Q_{\frac{5}{8}}} u^{p+m-\delta-2+\delta p}dz\leq \gamma.$$
Moreover, we choose a test function
$\varphi\in C^\infty (Q_{\frac{7}{8}})$, such that $0\leq\varphi\leq1$,
$$\varphi=0\quad\mathrm{on}\quad\partial Q_{\frac{7}{8}}\backslash K_{\frac{7}{8}}\times\left\{0\right\},\quad\varphi=1\quad\mathrm{on}\quad Q_{\frac{5}{8}}\quad\mathrm{and}\quad|D\varphi|+\left(-\frac{\partial \varphi}{\partial t}\right)_+\leq \gamma.$$
Applying Caccioppoli estimate \eqref{moser2}, with $\epsilon=-\delta$ and the test function $\varphi$ as defined above, the integral involving gradient is estimated by
\begin{equation*}\begin{split}\iint_{Q_{\frac{5}{8}}} u^{m-\delta-2}|Du|^{p}dz&\leq \iint_{Q_{\frac{7}{8}}} u^{m-\delta-2}|Du|^{p}\varphi^pdz
\\&\leq \gamma\iint_{Q_{\frac{7}{8}}}u^{m+p-\delta-2}|D\varphi|^pdz+\gamma\iint_{Q_{\frac{7}{8}}}u^{1-\delta}\left(-\frac{\partial \varphi}{\partial t}\right)_+dz.\end{split}\end{equation*}
The Lemma follows by Lemma \ref{uniform bound} and H\"older's inequality. We omit the details.
\end{proof}
The following proposition is our main result in this section.
\begin{proposition}\label{cold alt pro} There exist positive constants $\theta_0$, $S_2$ and $\eta_1$, depending only upon $m$, $p$, $N$, $C_0$ and $C_1$, such that if $u$ is a positive, weak supersolution in an open set which compactly contains $K_2\times(0,T_2]$,
\begin{equation}\begin{split}\label{theta0}\int_{K_{\frac{1}{2}}}u(x,0)dx\geq 2 \theta_0,\end{split}\end{equation}
and the condition
\eqref{cold assumption} holds
for all $t\in (0,1]$ and all level $k\geq1$, then
there holds
\begin{equation}\label{result1}u\geq \eta_1\quad\mathrm{a.\ e.\ \ in}\quad K_2\times\left(1+\tfrac{1}{2}S_2,S_2\right].\end{equation}
\end{proposition}
\begin{proof} In the inequality \eqref{weak form 22} we choose $\phi=\varphi(x)$ as a test function, where $\varphi\in C_0^\infty(K_{\frac{5}{8}})$,
$0\leq\varphi\leq1$, $\varphi=1$ in $K_{\frac{1}{2}}$ and $|D\varphi|\leq 8$. Then, we infer from \eqref{theta0} that
\begin{equation*}\begin{split}\label{weak form 2}
\int_{K_\frac{5}{8}}u(x,t)dx&\geq \int_{K_\frac{5}{8}}u(x,t)\varphi(x)dx\\&\geq\int_{K_\frac{5}{8}}u(x,0)\varphi(x)dx-\int^{(\frac{5}{8})^p}_{0}\int_{K_{\frac{5}{8}}}A(x,t,u,Du)\cdot D\varphi dxdt
\\&\geq
\int_{K_\frac{1}{2}}u(x,0)dx-\int^{(\frac{5}{8})^p}_{0}\int_{K_{\frac{5}{8}}}u^{m-1}|Du|^{p-1} |D\varphi| dxdt\\&\geq 2\theta_0-\tilde \gamma\geq \theta_0\end{split}\end{equation*}
for any $t\in\left(0,(\frac{5}{8})^p\right]$, provided $\theta_0\geq \tilde \gamma$. On the other hand, we use Lemma \ref{uniform bound} to obtain
$$\int^{(\frac{5}{8})^p}_{0}\int_{K_{\frac{5}{8}}}u(x,t)^qdxdt\leq \gamma$$
for some $q>1$ in the range \eqref{range}. Then, there exists a time level $t_1\in\left(0,(\frac{5}{8})^p\right]$ such that
$$\int_{K_\frac{5}{8}}u(x,t_1)dx\geq \theta_0\quad\quad\mathrm{and}\quad\quad\int_{K_{\frac{5}{8}}}u(x,t_1)^qdx\leq\gamma.$$
Hence, it follows from \cite[Lemma 5.9]{Kuusi} that
$$|[u(\cdot,t_1)\geq\tfrac{1}{2}\theta_0]\cap K_{\frac{5}{8}}|\geq \left(\tfrac{1}{2}\right)^{\frac{q}{q-1}}\gamma^{-\frac{1}{q-1}}|K_{\frac{5}{8}}|.$$
By Lemma \ref{refined}, there exist constants $\eta_0$, $\delta$ in $(0,1)$ and $C$, $\eta_*$, $d>0$ depending only upon $m$, $p$, $N$, $C_0$ and $C_1$, such that
\begin{equation*}u\geq \frac{1}{2}\eta_*\theta_0\quad\mathrm{a.\ e.}\quad\mathrm{in}\quad K_2\times\left(t_1+\frac{C^{p+m-3}}{(\eta_*\theta_0)^{p+m-3}}\frac{1}{2}\delta,t_1+\frac{C^{p+m-3}}{(\eta_*\theta_0)^{p+m-3}}\delta\right],\end{equation*}
where
$$\eta_*=\eta_0\left( \left(\tfrac{1}{2}\right)^{\frac{q}{q-1}}\gamma^{-\frac{1}{q-1}}\right)^{d+1}.$$
Therefore the estimate \eqref{result1} follows by choosing
\begin{equation}\label{T2 time}\eta_1=\tfrac{1}{2}\eta_*\theta_0\quad\mathrm{and}\quad S_2=\frac{C^{p+m-3}}{(\eta_*\theta_0)^{p+m-3}}.\end{equation}
This finishes the proof of Proposition \ref{cold alt pro}.
\end{proof}
\section{Proofs of the Main Results}
In this section we will prove the weak Harnack inequalities of positive, weak supersolutions to doubly nonlinear parabolic equations. The method of the proof has its origin in
\cite{Kuusi, DGV}. Since most of the arguments are standard by now, we will only sketch the proof and refer the read to \cite[chapter 3, section 13-14]{DGV} for the details.
\subsection{Proof of Theorem \ref{main result1}}
Let $\tilde \gamma$ be the constant defined as in \eqref{gradient uniform}. Let $(y,s)\in E_T$ be a fixed point. We set $\tilde M=\fint_{K_\rho(y)}u(x,s)dx$ and $M=(3\tilde\gamma)^{-1}\tilde M$. We
assume that $M>0$, otherwise there is nothing to prove. Moreover, we introduce the change of variables
$$x^\prime=\frac{x-y}{2\rho}, \quad t^\prime=M^{p+m-3}\frac{t-s}{\rho^p}\quad\mathrm{and}\quad v(x^\prime,t^\prime)=\frac{1}{M}u(y+2\rho x^\prime,s+M^{m+p-3}\rho^p t^\prime).$$
From \eqref{weak form deff}, it is easily seen that
\begin{equation*}\begin{split}
-\int_{-M^{m+p-3}s\rho^{-p}}^{M^{m+p-3}(T-s)\rho^{-p}}&\int_{K_8}v\frac{\partial \Phi }{\partial t^\prime}dx^\prime dt^\prime+ \int_{-M^{m+p-3}s\rho^{-p}}^{M^{m+p-3}(T-s)\rho^{-p}}\int_{K_8}\tilde A(x^\prime,t^\prime,v,D_{x^\prime}v)\cdot D_{x^\prime}\Phi\ dx^\prime dt^\prime\geq 0
\end{split}\end{equation*}
for any non-negative test function $\phi=\phi(x^\prime,t^\prime)$ with
$$\Phi\in C_0^\infty\left[K_8\times(-M^{m+p-3}s\rho^{-p},M^{m+p-3}(T-s)\rho^{-p})\right],$$
where
$$\tilde A(x^\prime,t^\prime,v,D_{x^\prime}v)=\frac{\rho^{p-1}}{M^{p+m-2}}A\left(y+2\rho x^\prime,s+M^{p+m-3}\rho^pt^\prime,Mv,\frac{M}{2\rho}D_{x^\prime}v\right).$$
From \eqref{structure def}, we deduce
\begin{equation*}
	\begin{cases}
	\tilde A(x^\prime,t^\prime,v,D_{x^\prime}v)\cdot D_{x^\prime}v\geq 2^{1-p}C_0|v|^{m-1}|Dv|^p,\\
	|\tilde A(x^\prime,t^\prime,v,D_{x^\prime}v)|\leq 2^{1-p}C_1|v|^{m-1}|Dv|^{p-1},
	\end{cases}
\end{equation*}
which implies that $v$ is a weak supersolution to the doubly degenerate parabolic equation
$$\partial_{t^\prime}v-\mathrm{div}\tilde A(x^\prime,t^\prime,v,D_{x^\prime}v)=0$$
in the cylinder $K_8\times\left(-M^{m+p-3}s\rho^{-p},M^{m+p-3}(T-s)\rho^{-p}\right)$. Furthermore, we observe that
\begin{equation}\label{MM} \fint_{K_\frac{1}{2}}v(x^\prime,0)dx^\prime=3\tilde\gamma.\end{equation}
Let $S_1$ and $S_2$ be as in \eqref{T1 time} and \eqref{T2 time} respectively. By reducing $\eta_0$, we assume, without loss of generality, that
$\min\{S_1,S_2\}\geq1$.
Furthermore, we need only consider the case
\begin{equation*}M^{p+m-3}\frac{T-s}{\rho^p}>\max\{S_1,S_2\}>1,\end{equation*}
since otherwise there is nothing to prove. We infer from \eqref{MM} that either Proposition \ref{hot pro} or Proposition \ref{cold alt pro} holds. By a suitable modification of the constants, there exist constants $\eta_*$
and $T_*$ depending only upon $m$, $p$, $N$, $C_0$ and $C_1$, such that
\begin{equation*}v\geq \eta_*\quad\mathrm{a.\ e.\ \ in}\quad K_2\times\left(1+\tfrac{1}{2}T_*,T_*\right].\end{equation*}
Transforming back to the original variables, we obtain the desired estimate \eqref{weak harnack local}. We omit the details.

\subsection{Proof of Theorem \ref{main result2}}\label{main results}
Without loss of generality,we may prove the theorem in the case that $(y,s)=(0,0)$. This is always possible by using a translation. It suffices to prove the theorem in the case that
\begin{equation}\label{uu}\fint_{K_\rho}u(x,0)dx>2c\left(\frac{\rho^p}{T^\prime}\right)^{\frac{1}{p+m-3}}.\end{equation}
We set
$v(x,t)=u\left(\delta\rho x,(\delta\rho)^pt\right),$
where $\delta>1$ is to be chosen. Then $v$ is a positive, weak supersolution in $\mathbb{R}^N\times (0,T^\prime(\delta\rho)^{-p}]$.
Let $c$ and $\gamma$ be as in Theorem \ref{main result1}.
We need only consider the case
\begin{equation}\label{vv}\fint_{K_1}v(x,0)dx>2c\left(\frac{(\delta\rho)^p}{T}\right)^{\frac{1}{p+m-3}}.\end{equation}
We now use Theorem \ref{main result1} for $v$ to obtain
\begin{equation*}\fint_{K_1}v(x,0)dx\leq c\left(\frac{(\delta\rho^p)}{T}\right)^{\frac{1}{p+m-3}}+\gamma\essinf_{Q_*}v,\end{equation*}
where
$$Q_*=K_*\times(\tfrac{1}{2}T_*,T_*)\quad\mathrm{and}\quad T_*=c^{m+p-3}\left(\frac{1}{2}\fint_{K_1}v(x,0)dx\right)^{3-m-p}<\frac{T}{(\delta p)^p}.$$
It follows from \eqref{vv} that
\begin{equation}\label{vvv}\fint_{K_1}v(x,0)dx\leq 2\gamma \essinf_{Q_*}v.\end{equation}
At this point, we choose
\begin{equation}\label{hhhhhh}\delta^p=c^{3-m-p}\frac{T^\prime}{\rho^p}\left(\frac{1}{2}\fint_{K_1}v(x,0)dx\right)^{p+m-3}.\end{equation}
Moreover, the estimate \eqref{vv} yields
$$\delta^\lambda\geq c^{3-m-p}\frac{T^\prime}{\rho^p}\left(\frac{1}{2}\fint_{K_\rho}u(x,0)dx\right)^{p+m-3},$$
where $\lambda=N(p+m-3)+p$.
This inequality together with \eqref{uu} ensures $\delta>1$. Finally, we conclude from \eqref{vv}-\eqref{hhhhhh} that
\begin{equation}\label{uuu}\fint_{K\rho}u(x,0)dx\leq 2\gamma \delta^N\inf_{Q^\prime}u \leq 2\gamma (c\gamma)^{\frac{N(p+m-3)}{p}}\left(\frac{T^\prime}{\rho^p}\right)^{\frac{N}{p}}\left(\essinf_{Q^\prime}u\right)^{\frac{N(p+m-3)}{p}}\end{equation}
where $Q^\prime=K_{4\delta\rho}\times (\tfrac{1}{2}T^\prime,T^\prime)$.
This finishes the proof, the detailed verification of \eqref{uuu} being left to the reader.
\begin{remark}\end{remark}
It would be interesting to establish similar results of Theorem \ref{main result1} and Theorem \ref{main result2} for the case $m+p>3$ and $0<p<1$.
\section*{Acknowledgement}
The author wishes to thank Ugo Gianazza, Tuomo Kuusi, Pekka Lehtel\"a, Peter Lindqvist and Vincenzo Vespri for the valuable discussions.

\bibliographystyle{abbrv}

\end{document}